\documentclass[a4paper,10pt]{amsart}
\usepackage{enumerate, amsmath, amsfonts, amssymb, amsthm, mathtools, thmtools, wasysym, graphics, graphicx, xcolor, frcursive,xparse,comment,ytableau}
\usepackage[all]{xy}
\usepackage{hyperref}

\usepackage{tikz}
\usetikzlibrary{calc,through,backgrounds,shapes,matrix}

\definecolor{darkblue}{rgb}{0.0,0,0.7} 
\newcommand{\darkblue}{\color{darkblue}} 
\definecolor{darkred}{rgb}{0.7,0,0} 
\definecolor{lightgrey}{rgb}{0.7,0.7,0.7} 

\def\Z{\mathbb{Z}}
\def\R{\mathbb{R}}

\def\simp{\Delta}
\def\wa{\widetilde{W}}
\def\ac{\mathcal{A}}
\def\Q{\check{Q}}
\def\bijQtoW{\mathsf{crt}}

\def\HH{\mathcal{H}}

\def\ar{\widetilde{\Phi}}

\def\h{\mathrm{ht}}
\def\cw{\check{\omega}}

\def\rhog{\frac{\rho}{g}}
\def\rhoh{\frac{\check{\rho}}{h}}
\def\quadr{Q}
\def\rat{b}

\newcommand{\cwl}{\check{\Lambda}}
\newcommand{\wex}{\wa_{\mathrm{ex}}}
\newcommand{\cycl}{\Omega}

\newcommand{\Expt}[2]{\operatorname*{\mathbb{E}}\limits_{#1}(#2)}
\newcommand{\Var}[2]{\operatorname*{\mathbb{V}}\limits_{#1}(#2)}
\newcommand{\core}{{\sf core}}
\newcommand{\pent}{{\sf pent}}
\newcommand{\bino}{{\sf binom3}}
\newcommand{\size}{{\sf size}}
\newcommand{\sizeshift}{{\sf zise}}
\newcommand{\zise}{{\sf zise}}
\newcommand{\Sommers}{\mathcal{S}}
\newcommand{\hgt}{{\sf ht}}
\newcommand{\comp}{{\sf comp}}
\newcommand{\lcm}{{\sf lcm}}

\newcommand{\inv}{{\sf inv}}
\newcommand{\bijtocore}{{\sf core}_{\mathfrak{S}}}
\newcommand{\bijqtocore}{{\sf core}_{\check{Q}}}
\renewcommand{\mod}{\operatorname{mod}}

\newcommand{\waf}{\widetilde{w}}
\newcommand{\wf}{\waf_{\rat}}
\newcommand{\affr}{\tilde{\alpha}}
\newcommand{\amax}{\tilde{\alpha}}
\newtheorem{theorem}{Theorem}[section]
\newtheorem{proposition}[theorem]{Proposition}
\newtheorem{corollary}[theorem]{Corollary}
\newtheorem{lemma}[theorem]{Lemma}

\theoremstyle{definition}
\newtheorem{definition}[theorem]{Definition}
\newtheorem{example}[theorem]{Example}
\newtheorem{conjecture}[theorem]{Conjecture}
\newtheorem{openproblem}[theorem]{Open Problem}
\newtheorem{remark}[theorem]{Remark}

\usepackage{cleveref}
\usepackage[all]{xy}
\usepackage[T1]{fontenc}
\crefformat{footnote}{#2\footnotemark[#1]#3}
\crefformat{conjecture}{Conjecture~#2#1#3}

\newcommand{\defn}[1]{\emph{\darkblue #1}}

\title[Strange Expectations]
  {Strange Expectations}

\author[M.~Thiel]{Marko Thiel}
\address[M.~Thiel]{Fakult\"{a}t f\"{u}r Mathematik, Universit\"{a}t Wien, Vienna, Austria}
\email{marko.thiel@univie.ac.at}

\author[N.~Williams]{Nathan Williams}
\address[N.~Williams]{LaCIM, Universit\'e de Qu\'ebec \`a Mont\'eal \\
Montr\'eal (Qu\'ebec), Canada}
\email{nathan.f.williams@gmail.com}

\date{\today}

\begin{document}
   
\begin{abstract}
Let $\gcd(a,b)=1$.  J.~Olsson and D.~Stanton proved that the maximum number of boxes in a simultaneous $(a,b)$-core is
	\[\max_{\lambda \in \core(a,b)} (\size(\lambda)) = \frac{(a^2-1)(b^2-1)}{24},\]
and that this maximum was achieved by a unique core.  P.~Johnson combined Ehrhart theory with the polynomial method to prove D.~Armstrong's conjecture that the expected number of boxes in a simultaneous $(a,b)$-core is
	\[\Expt{\lambda \in \core(a,b)}{\size(\lambda)} = \frac{(a-1)(b-1)(a+b+1)}{24}.\]
We extend P.~Johnson's method to compute the variance to be
    \[\Var{\lambda \in \core(a,b)}{\size(\lambda)} = \frac{ab(a-1)(b-1)(a+b)(a+b+1)}{1440}.\]
By extending the definitions of ``simultaneous cores'' and ``number of boxes'' to affine Weyl groups, we give uniform generalizations of all three formulae above to simply-laced affine types.  We further explain the appearance of the number $24$ using the ``strange formula'' of H.~Freudenthal and H.~de Vries.


%
\end{abstract}

\maketitle

\section{Introduction}
 

\subsection{Motivation: Simultaneous Cores}
An \defn{$a$-core} is an integer partition with no hook-length divisible by $a$.  As a first example, observe that the $2$-cores are exactly those partitions of staircase shape.  According to the notes in G.~James and A.~Kerber~\cite{james1981representation}, cores were originally developed by T.~Nakayama in his study of the modular representation theory of the symmetric group~\cite{nakayama1940some}.\footnote{The relationship arises as follows.  The irreducible representations of $\mathfrak{S}_n$ over a field of characteristic zero are parametrized by integer partitions of $n$.  T.~Nakayama conjectured that the $p$-blocks for $\mathfrak{S}_n$ are in bijection with $p$-cores---more specifically, that the $p$-block corresponding to a $p$-core $\lambda$ contains exactly those representations whose indexing partitions have core $\lambda$~\cite{nakayama1940some}.  This conjecture was proven by R.~Brauer and G.~Robinson~\cite{brauer1947conjecture}.}  For $\lambda$ a partition of $k$, we write $\size(\lambda):=k$.  Let $\core(a)$ be the set of all $a$-cores, and define \[\core_k(a) := \{ \lambda \in \core(a) : \size(\lambda)=k \}.\]  The following identity relating integer partitions and $a$-cores is a fun exercise using the abacus:
\begin{theorem}[{Generating function for $\size$ on $\core(a)$;~\cite{james1981representation,garvan1990cranks}}]
	\[\sum_{k=0}^\infty |\core_k(a)| q^k=\prod_{i=1}^\infty \frac{(1-q^{ai})^a}{1-q^i}.\]
    \label{thm:total_number}
   \end{theorem}

An \defn{$(a,b)$-core} is a partition that is both an $a$-core and a $b$-core.  We denote the set of $(a,b)$-cores by $\core(a,b)$.  When $a$ and $b$ are coprime, J.~Anderson proved the surprising fact that there are only finitely many $(a,b)$-cores.

\begin{theorem}[{Number of simultaneous $(a,b)$-cores; J.~Anderson~\cite{anderson2002partitions}}]
$ $\newline
	For $\gcd(a,b)=1$, \[\left|\core(a,b)\right| = \frac{1}{a+b}\binom{a+b}{b}.\]
\label{thm:anderson_number}
\end{theorem}

In part due to their connection with rational Dyck paths (and hence diagonal harmonics and the zeta map)~\cite{sulzgruber2014type,ceballos2015combinatorics,bergeron2014compositional} and rational Catalan combinatorics~\cite{armstrong2013rational,gorsky2014affine,armstrong2014rational}, simultaneous cores have recently attracted attention.  Furthermore, the study of simultaneous cores has now transcended these original motivations, and they have become combinatorial objects worthy of study in their own right~\cite{nath2008t,aukerman2009simultaneous,fayers2011t,amdeberhan2014multi,yang2014enumeration, nath2014symmetry,chen2014average, aggarwal2014converse,fayers2014generalisation,xiong2014largest,aggarwal2015does,fayers2015s}.  In this direction, there are two main results on the statistic $\size$.

\begin{theorem}[{Maximum $\size$ of an $(a,b)$-core; J.~Olsson and D.~Stanton~\cite{olsson2007block}}]
	For $\gcd(a,b)=1$, \[\max_{\lambda \in \core(a,b)}(\size(\lambda)) = \frac{(a^2-1)(b^2-1)}{24}.\]  This maximum is attained by a unique $(a,b)$-core.
\label{thm:olsson_stanton_max_size}
\end{theorem}

A stronger statement is actually true: J.~Vandehey proved that the diagram of this unique $(a,b)$-core maximizing $\size$ contains the diagrams of all other $(a,b)$-cores~\cite{vandehey2008general,fayers2011t}~\cite[Remark 4.11]{olsson2007block}.

D.~Armstrong conjectured the following attractive formula~\cite{armstrong2011conjecture,armstrong2014results}, which was proven for $b=a+1$ by R.~Stanley and F.~Zanello~\cite{stanley2013catalan}; for $b=ma+1$ by A.~Aggarwal~\cite{aggarwal2014armstrong}; and in full generality by P.~Johnson~\cite{johnson2015lattice}.

\begin{theorem}[{Expected $\size$ of an $(a,b)$-core; P.~Johnson~\cite{johnson2015lattice}}]
$ $\newline
	For $\gcd(a,b)=1$, \[\Expt{\lambda \in \core(a,b)}{\size(\lambda)} = \frac{(a-1)(b-1)(a+b+1)}{24}.\]
\label{thm:johnson_expected_size}
\end{theorem}

Our first new result is to extend P.~Johnson's technique\footnote{This is of course the well-known \emph{Paulynomial} method.} to compute the variance of $\core(a,b)$.

\begin{theorem}[{Variance of $\size$ on $(a,b)$-cores}] 
$ $\newline
For $\gcd(a,b)=1$, \[\Var{\lambda \in \core(a,b)}{\size(\lambda)} = \frac{ab(a-1)(b-1)(a+b)(a+b+1)}{1440}.\]
\label{thm:us_variance}
\end{theorem}

With more effort, we also compute the third moment, which was conjectured by D.~Armstrong in 2013~\cite{drewprivate}.

\begin{theorem}[{Third moment of $\size$ on $(a,b)$-cores}] 
$ $\newline
For $\gcd(a,b)=1$, let $\mu:=\Expt{\lambda \in \core(a,b)}{\size(\lambda)}$.  Then \tiny\[\sum_{\lambda \in \core(a,b)} \left(\size(\lambda)-\mu\right)^3 = \frac{ab(a-1) (b-1) (a+b) (a+b+1) \left(2 a^2 b-3 a^2+2 a b^2-3 a b-3 b^2-3 \right)}{60480}.\]\normalsize
\label{thm:us_third_moment}
\end{theorem}

D.~Armstrong also conjectured a formula for the fourth cumulant on the basis of extensive computations.  We will not state or prove his conjecture, but the interested reader might enjoy~\Cref{sec:integrals}.

\medskip

\subsection{Simply-Laced Generalizations}

The main purpose of this paper is to give generalizations of~\Cref{thm:total_number,thm:anderson_number,thm:olsson_stanton_max_size,thm:johnson_expected_size,thm:us_variance} for all simply-laced types.  This simply-laced requirement arises from a simplification that only happens in those types, and is explained in~\Cref{sec:simply_laced}.   

To this end, we fix the following notation, which is fully reviewed in~\Cref{sec:affine_weyl_groups}. Let $\Phi$ be an irreducible crystallographic root system of rank $n$ with ambient space $V$ and Weyl group $W$.
Let $\ar$ be the set of affine roots and denote the affine Weyl group by $\wa$. Choose a set of \defn{simple roots} $\Delta$ for $\Phi$ and let $\Phi^+$ be the corresponding set of \defn{positive roots}. Say $\Phi$ has \defn{exponents} $e_1 \leq e_2 \leq \cdots \leq e_n$, \defn{Coxeter number} $h:=e_n+1$, and \defn{dual Coxeter number} $g$.  \Cref{thm:total_number,thm:anderson_number,thm:olsson_stanton_max_size,thm:johnson_expected_size,thm:us_variance} will be recovered in this notation by specializing to type $A_{a-1}$, in which case $n=a-1$, $h=a$ and $\widetilde{W}=\widetilde{\mathfrak{S}}_{a}$.

A useful analogue of a core for $\widetilde{W}$ turns out to be a point of the coroot lattice $\check{Q}$, which we emphasize with the notation \[ \core(\widetilde{W}) := \check{Q}.\] In~\Cref{sec:affine_sym_and_cores}, we recall how this definition recovers $a$-cores when $\widetilde{W}=\widetilde{\mathfrak{S}}_{a}$.

\medskip

In order to generalize~\Cref{thm:total_number,thm:olsson_stanton_max_size,thm:johnson_expected_size,thm:us_variance}, we require a notion of the statistic $\size$, defined combinatorially for $\widetilde{\mathfrak{S}}_{a}$ using the Ferrers diagram of a core.  For the purposes of the introduction, we pull this out of a hat (but see~\Cref{def:size_on_elements} and~\Cref{ex:heights}): for any point $x \in V$---and in particular for any point in the coroot lattice $\check{Q}$---let

\[\size(x):=\frac{g}{2}||x||^2-\langle x,\rho\rangle,\]
where $\rho:=\frac{1}{2}\sum_{\alpha\in\Phi^+}\alpha$ is half the sum of all positive roots.

The statistic $\size$ on $\core(\widetilde{\mathfrak{S}}_{a})$ recovers the number of boxes in the corresponding core in $\core(a)$ (\Cref{def:size_and_quadr} and~\Cref{prop:size_in_type_a}).

\medskip

Define \[\core_k(\widetilde{W}) := \{ \lambda \in \core(\widetilde{W}) : \size(\lambda)=k \}.\]  By specializing his character formula at a primitive $h$th root of unity, I.~G.~MacDonald has uniformly generalized~\Cref{thm:total_number} to all simply-laced types.

\begin{theorem}[\protect{\cite[Theorem 8.16]{macdonald1971affine}}]
For $\widetilde{W}$ the affine Weyl group of a simply-laced irreducible crystallographic root system $\Phi$ with Weyl group $W$, let $f(q)$ be the characteristic polynomial of a Coxeter element in $W$ (in the reflection representation). Then\footnote{\label{footnote:macdonald}The last equality in~\cite[Theorem 8.16]{macdonald1971affine} appears to have a small typo.}

\[\sum_{k=0}^\infty \left|\core_k(\widetilde{W})\right|q^k = \prod_{i=1}^\infty \left(f(q^i)(1-q^{hi})^n\right).\] 
\label{thm:macdonald_number}
\end{theorem}

\medskip

Having generalized the notion of core and the statistic $\size$, we still require a definition of simultaneous cores.  In~\Cref{sec:dilations_and_cores}, for any positive integer $b$ that is relatively prime to $h$, we define the \defn{Sommers region}
\begin{multline*}
 \Sommers_\Phi(b):=\{x\in V:\langle x,\alpha\rangle\geq-t\text{ for all }\alpha\in \Phi_r\text{ and }
 \langle x,\alpha\rangle\leq t+1\text{ for all }\alpha\in \Phi_{h-r}\}.
\end{multline*}
This is the region in $V$ bounded by all the affine hyperplanes corresponding to affine roots of height $b$.

M.~Haiman has uniformly proven (for \emph{all} affine Weyl groups) the following generalization of~\Cref{thm:anderson_number}~\cite{haiman1994conjectures}, which we state in terms of $\Sommers_{\Phi}(b)$ using a result of E.~Sommers~\cite[Theorem 5.7]{sommers2005b}.  We remark that R.~Suter has also (presumably independently) observed essentially the same formula type-by-type~\cite{suter1998number}.

\begin{theorem}[{Number of $(\widetilde{W},b)$-cores; M.~Haiman~\cite{haiman1994conjectures}}]
$ $\newline
	For $\gcd(h,b)=1$, \[\left|\core(\widetilde{W},b)\right| = \frac{1}{|W|} \prod_{i=1}^n (b+e_i).\]
\label{thm:haiman_number}
\end{theorem}

We now state generalizations of~\Cref{thm:olsson_stanton_max_size,thm:johnson_expected_size,thm:us_variance}.

\begin{theorem}[{Maximum $\size$ of a $(\widetilde{W},b)$-core}]
$ $\newline
	For $\widetilde{W}$ a simply-laced affine Weyl group with $\gcd(h,b)=1$, \[\max_{\lambda \in \core(\widetilde{W},b)}(\size(\lambda)) = \frac{n(b^2-1)(h+1)}{24}.\]  This maximum is attained by a unique $\lambda \in \core(\widetilde{W},b)$.
\label{thm:max_size}
\end{theorem}

In~\Cref{conj:w_b_is_maximal_in_weak}, we conjecture an analogue of J.~Vandehey's result for $(a,b)$-cores, using the inversion sets of the dominant affine elements corresponding to $(\widetilde{W},b)$-cores.

\begin{theorem}[{Expected $\size$ of a $(\widetilde{W},b)$-core}]
$ $\newline
	For $\widetilde{W}$ a simply-laced affine Weyl group with $\gcd(h,b)=1$, \[\Expt{\lambda \in \core(\widetilde{W},b)}{\size(\lambda)} = \frac{n(b-1)(h+b+1)}{24}.\]
\label{thm:expected_size}
\end{theorem}

The appearence of the number $24$ in~\Cref{thm:max_size,thm:expected_size} is explained by~\Cref{thm:explicit_quadr} and~\Cref{def:size_and_quadr}, where we relate the statistic $\size$ to an easily-computed quadratic form $\quadr$ whose value at $0$ is $-\frac{\langle \rho,\rho \rangle}{2g}$.  By the ``strange formula'' of H.~Freudenthal and H.~de Vries (\Cref{thm:identities}),  \[-\frac{\langle \rho,\rho \rangle}{2g}=-\frac{n(h+1)}{24},\] which accounts for the constant term in both theorems.

We also compute a uniform formula for the variance $\mathbb{V}$ of the statistic $\size$ over $\core(\widetilde{W},b)$.

\begin{theorem}[{Variance of $\size$ on $(\widetilde{W},b)$-cores}]
$ $\newline
	For $\widetilde{W}$ a simply-laced affine Weyl group with $\gcd(h,b)=1$, \[\Var{\lambda \in \core(\widetilde{W},b)}{\size(\lambda)} = \frac{nhb(b-1)(h+b)(h+b+1)}{1440}.\]
\label{thm:variance}
\end{theorem}


\begin{remark}
For affine types outside of $\widetilde{A}_n$, we did not compute any moments beyond the second (but see~\Cref{sec:integrals}).  Our justification is that we verified that there is no possible assignment of $a \mapsto \{n+1,h\}$ in~\Cref{thm:us_third_moment}---where each factor of $a$ is assigned independently---that results in a uniform product formula simultaneously valid for all simply-laced affine Weyl groups.  This leaves open the possibility that there are ``hidden'' factors of powers of $\frac{n+1}{h}$, though we suspect that this is not the case.
\label{rem:no_uniform}
\end{remark}

We stress that although the statements of~\Cref{thm:max_size,thm:expected_size,thm:variance} are uniform for simply-laced types, many of our proofs (especially the computations in~\Cref{sec:proofs}) are very much type-dependent.   It would be desirable to have uniform proofs.

\medskip

\subsection{Proof Strategy and Summary}
\label{sec:technical_difficulties}

We outline here the two technical difficulties (both already present in the type $A_n$ case studied in~\cite{johnson2015lattice}), the explanation and resolution of which will occupy much of~\Cref{sec:hyp_and_weyl,sec:dilations_and_cores,sec:statistics}.  Given a vector space $V$ and an $n$-dimensional polytope $P$ in $V$ whose vertices are elements of a lattice $L$---that is, $P$ is an \defn{integer polytope} with respect to $L$---Ehrhart theory tells us that the number of lattice points of $L$ inside the $b$-th dilation of $P$ is given by a polynomial $\mathcal{P}^L(b)$ of degree $n$ in $b$.  Ehrhart theory extends to Euler-Maclaurin theory (see~\Cref{sec:weighted_ehrhart}), which says that given a polynomial $p$ on $V$ of degree $m$, the sum
\[\mathcal{P}^L_p(b):=\sum_{x \in b \mathcal{P} \cap L} p(x).\]
over these lattice points gives a polynomial $\mathcal{P}^L_p(b)$ of degree $n+m$ in $b$.

To prove~\Cref{thm:us_third_moment,thm:expected_size,thm:variance}, we wish to use Ehrhart theory combined with the polynomial method to determine \[\sum_{\lambda \in \Sommers_\Phi(b) \cap \check{Q}} \size^i(\lambda) \text{ for } i=1,2,3.\]
  
  The trouble is that Ehrhart theory manifestly does not apply: $\Sommers_\Phi(b)$ is \emph{neither} the dilation of a polytope, \emph{nor} are its vertices in the coroot lattice $\check{Q}$ for general values of $b$.
  
The first obstacle is that $\Sommers_\Phi(b)$ is not the \emph{dilation} of a polytope---as the residue class of $b \mod h$ changes, so does the orientation of $\Sommers_\Phi(b)$. We therefore first translate the study of $\Sommers_\Phi(b)$ to the study of $b\ac$---the $b$-fold dilation of the fundamental alcove---which remains in a fixed orientation as $b$ varies:
     \begin{itemize}
		\item \Cref{thm:wf} uniformly proves that $\Sommers_\Phi(b)$ may be mapped bijectively to $b\ac$ via an explicit rigid motion $\wf$ (filling a gap in the literature); and
         \item    Using the rigid motion $\wf$, we translate the statistic $\size$ on $\core(\widetilde{W},b)$ onto a statistic $\sizeshift$ on $b\ac \cap \check{Q}$ in~\Cref{cor:size_on_ba}.
     \end{itemize}

The second obstacle is that $\Sommers_\Phi(b)$---and therefore also $b\ac$---is not an \emph{integer polytope} with respect to the coroot lattice $\check{Q}$.  Following P.~Johnson, Ehrhart theory extends to rational polytopes, at the cost of trading polynomiality for quasipolynomiality (with an explicit period).  It is somewhat easier to translate the study of the coroots $b\ac\cap \check{Q}$ to the study of the coweights $b\ac \cap \check{\Lambda}$:
    \begin{itemize}
        \item We recall in~\Cref{prop:barational} that the polytope $bA_0$ is a rational polytope in the coweight lattice $\check{\Lambda}$;
        \item The coroot lattice $\check{Q}$ is a lattice of index $f \in \mathbb{N}$ (the \defn{index of connection}) inside $\check{\Lambda}$.  We define the group $\cycl=\check{\Lambda}/\check{Q}$ in~\Cref{sec:symmetry_of_digram}, and we prove in~\Cref{thm:cycl_action_and_Q} that each $b\cycl$-orbit of $b\ac \cap\check{\Lambda}$ contains exactly one point of $b\ac \cap \check{Q}$; and
        \item We show in~\Cref{thm:cycl_action_and_quad} that the action of $b\cycl$ preserves $\size$.
    \end{itemize}
\medskip

The remainder of this paper is structured as follows.  In~\Cref{sec:affine_weyl_groups} we review the basic notions of finite and affine Weyl groups.  In~\Cref{sec:affine_sym_and_cores}, we review how $a$-cores fit into the framework of affine Weyl groups as the special case $\widetilde{W}=\widetilde{\mathfrak{S}}_a$.  In~\Cref{sec:dilations_and_cores}, we generalize $a$-cores to $\widetilde{W}$ using the Sommers region $\Sommers_\Phi(b)$, and we relate $\Sommers_\Phi(b)$ and $b\ac$.  We also recall M.~Haiman's~\Cref{thm:haiman_number} and prove~\Cref{thm:max_size}.  In~\Cref{sec:statistics}, we generalize the statistic $\size$ to $\Sommers_\Phi(b)$ for all affine Weyl groups, and we study how it transforms to a statistic on $b\ac$.  In~\Cref{sec:proofs} for $b$ coprime to $h$, we compute the relevant residue classes of the Ehrhart quasipolynomial $\ac^{\check{\Lambda}}_{\zise^i}(b)$ to conclude~\Cref{thm:us_variance,thm:us_third_moment,thm:expected_size,thm:variance}.  In~\Cref{sec:open_problems}, we state some open problems and conjectures regarding higher moments, non-simply-laced types, and combinatorial models.

\section{Affine Weyl Groups}
\label{sec:affine_weyl_groups}

In this section, we introduce finite and affine root systems (\Cref{sec:root_systems,sec:affine_roots}) and associated data.  We also define their associated hyperplane arrangements and Weyl groups (\Cref{sec:hyp_and_weyl}).  Finally, we define the abelian group $\cycl$, which allows us to relate the coroot and coweight lattices in~\Cref{thm:cycl_action_and_Q}.

\subsection{Root Systems}
\label{sec:root_systems}
Let $\Phi$ be an irreducible crystallographic root system of rank $n$ with ambient space $V$.
Define the \defn{root lattice} $Q$ of $\Phi$ as the lattice in $V$ generated by $\Phi$.
Let $\Phi^+$ be a system of \defn{positive roots} for it and let $\simp=\{\alpha_1,\alpha_2,\ldots,\alpha_n\}$ be the corresponding system of \defn{simple roots}.
Then $\Phi$ is the disjoint union of $\Phi^+$ and $-\Phi^+$, and $\simp$ is a basis for $V$.

For $\alpha\in\Phi$, we may write $\alpha$ in the basis of simple roots as $\alpha=\sum_{i=1}^n a_i\alpha_i$, where the coefficients $a_i$ are either all nonnegative or all nonpositive.  We define the \defn{height} of $\alpha$ as the sum of the coefficients:
$\h(\alpha):=\sum_{i=1}^n a_i$.
Notice that $\h(\alpha)>0$ if and only if $\alpha\in\Phi^+$ and $\h(\alpha)=1$ if and only if $\alpha\in\simp$.
There is a unique root \[\amax =\sum_{i=1}^n c_i\alpha_i \in\Phi\] of maximal height, which we call the \defn{highest root} of $\Phi$. We choose to normalize the inner product $\langle\cdot,\cdot\rangle$ on $V$ in such a way that $\|\amax\|^2=2$.
We define the \defn{Coxeter number} of $\Phi$ as $h:=1+\h(\amax)=1+\sum_{i=1}^n c_i$.
Let $\rho:=\frac{1}{2}\sum_{\alpha\in\Phi^+}\alpha$.

For a root $\alpha\in\Phi$, define its \defn{coroot} as $\check{\alpha}:=\frac{2\alpha}{\|\alpha\|^2}$.
Define the \defn{dual root system} of $\Phi$ as $\Phi^{\vee}:=\left\{\check{\alpha}:\alpha\in\Phi\right\}$. It is itself an irreducible crystallographic root system.
We say that $\Phi$ is \defn{simply-laced} if all roots $\alpha\in\Phi$ satsify $\|\alpha\|^2=2$. So in this case $\check{\alpha}=\alpha$ for all $\alpha\in\Phi$ and thus $\Phi=\Phi^{\vee}$.

Define the \defn{coroot lattice} $\Q$ of $\Phi$ as the lattice in $V$ generated by $\Phi^{\vee}$.
Let $\check{\rho}:=\frac{1}{2}\sum_{\alpha\in\Phi^+}\check{\alpha}$.

Finally, let $(\cw_1,\cw_2,\ldots,\cw_n)$ be the basis that is dual to the basis $(\alpha_1,\alpha_2,\ldots,\alpha_n)$ of $V$ consisting of the simple roots, so that $\langle \cw_i,\alpha_j\rangle=\delta_{i,j}$.
Then $\cw_1,\cw_2,\ldots,\cw_n$ are the \defn{fundamental coweights}. They are a basis of the coweight lattice 
\[\cwl:=\{x\in V:\langle x,\alpha\rangle\in\Z\text{ for all }\alpha\in\Phi\}\]
of $\Phi$. We also have
\[\check{\rho}=\sum_{i=1}^n\cw_i,\]
so that $\langle\check{\rho},\alpha\rangle=1$ for all $\alpha\in\simp$ and thus $\langle\check{\rho},\alpha\rangle=\h(\alpha)$ for all $\alpha\in\Phi$.

We can write the highest root $\amax$ (which is its own coroot) in terms of the coroots corresponding to the simple roots:
\[\amax=\sum_{i=1}^n d_i\check{\alpha_i}.\]
Then we define the \defn{dual Coxeter number} of $\Phi$ as $g:=1+\sum_{i=1}^nd_i$.

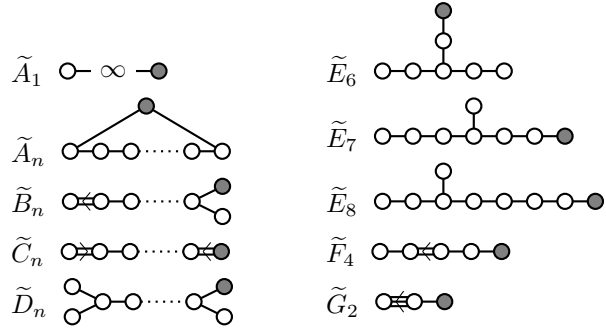
\begin{figure}[tbp]
\begin{center}
\begin{tabular}{cc}
 \parbox[t]{.3\textwidth}{
  \begin{tikzpicture}[scale=.2]
    \draw (-1,0) node[anchor=east]  {$\widetilde{A}_1$};
    \draw[thick] (0 cm,0) -- (6 cm,0) node [midway,fill=white] {$\infty$};
    \draw[thick,solid,fill=white] (0cm,0) circle (.5cm);
    \draw[thick,solid,fill=gray] (6cm,0) circle (.5cm);
  \end{tikzpicture}
  
  \begin{tikzpicture}[scale=.2]
    \draw (-1,0) node[anchor=east]  {$\widetilde{A}_n$};
    \draw[thick] (0 cm,0) -- (4 cm,0);
    \draw[dotted,thick] (4 cm,0) -- (8 cm,0);
    \draw[thick] (8 cm,0) -- (10 cm,0);
    \draw[thick] (0 cm,0) -- (5 cm,3);
    \draw[thick] (5 cm,3) -- (10 cm,0);
    \draw[thick,solid,fill=white] (0cm,0) circle (.5cm);
    \draw[thick,solid,fill=white] (2cm,0) circle (.5cm);
    \draw[thick,solid,fill=white] (4cm,0) circle (.5cm);
    \draw[thick,solid,fill=white] (8cm,0) circle (.5cm);
    \draw[thick,solid,fill=white] (10cm,0) circle (.5cm);
    \draw[thick,solid,fill=gray] (5cm,3) circle (.5cm);
  \end{tikzpicture}
  
  \begin{tikzpicture}[scale=.2]
    \draw (-1,0) node[anchor=east]  {$\widetilde{B}_n$};
    \draw[thick] (0 cm,-.2) -- (2 cm,-.2);
    \draw[thick] (0 cm,.2) -- (2 cm,.2);
    \draw[thin] (1.3 cm,.6) -- (.8 cm,0);
    \draw[thin] (1.3 cm,-.6) -- (.8 cm,0);
    \draw[thick] (2 cm,0) -- (4 cm,0);
    \draw[dotted,thick] (4 cm,0) -- (8 cm,0);
    \draw[thick] (8 cm,0) -- (10 cm,1);
    \draw[thick] (8 cm,0) -- (10 cm,-1);
    \draw[thick,solid,fill=white] (0cm,0) circle (.5cm);
    \draw[thick,solid,fill=white] (2cm,0) circle (.5cm);
    \draw[thick,solid,fill=white] (4cm,0) circle (.5cm);
    \draw[thick,solid,fill=white] (8cm,0) circle (.5cm);
    \draw[thick,solid,fill=gray] (10cm,1) circle (.5cm);
    \draw[thick,solid,fill=white] (10cm,-1) circle (.5cm);
  \end{tikzpicture}
  
  \begin{tikzpicture}[scale=.2]
    \draw (-1,0) node[anchor=east]  {$\widetilde{C}_n$};
    \draw[thick] (0 cm,-.2) -- (2 cm,-.2);
    \draw[thick] (0 cm,.2) -- (2 cm,.2);
    \draw[thin] (.7 cm,.6) -- (1.2 cm,0);
    \draw[thin] (.7 cm,-.6) -- (1.2 cm,0);
    \draw[thick] (2 cm,0) -- (4 cm,0);
    \draw[dotted,thick] (4 cm,0) -- (8 cm,0);
    \draw[thick] (8 cm,-.2) -- (10 cm,-.2);
    \draw[thick] (8 cm,.2) -- (10 cm,.2);
    \draw[thin] (9.3 cm,.6) -- (8.8 cm,0);
    \draw[thin] (9.3 cm,-.6) -- (8.8 cm,0);
    \draw[thick,solid,fill=white] (0cm,0) circle (.5cm);
    \draw[thick,solid,fill=white] (2cm,0) circle (.5cm);
    \draw[thick,solid,fill=white] (4cm,0) circle (.5cm);
    \draw[thick,solid,fill=white] (8cm,0) circle (.5cm);
    \draw[thick,solid,fill=gray] (10cm,0) circle (.5cm);
  \end{tikzpicture}

  \begin{tikzpicture}[scale=.2]
    \draw (-1,0) node[anchor=east]  {$\widetilde{D}_n$};
    \draw[thick] (0 cm,-1) -- (2 cm,0);
    \draw[thick] (0 cm,1) -- (2 cm,0);
    \draw[thick] (2 cm,0) -- (4 cm,0);
    \draw[dotted,thick] (4 cm,0) -- (8 cm,0);
    \draw[thick] (8 cm,0) -- (10 cm,1);
    \draw[thick] (8 cm,0) -- (10 cm,-1);
    \draw[thick,solid,fill=white] (0cm,-1) circle (.5cm);
    \draw[thick,solid,fill=white] (0cm,1) circle (.5cm);
    \draw[thick,solid,fill=white] (2cm,0) circle (.5cm);
    \draw[thick,solid,fill=white] (4cm,0) circle (.5cm);
    \draw[thick,solid,fill=white] (8cm,0) circle (.5cm);
    \draw[thick,solid,fill=gray] (10cm,1) circle (.5cm);
    \draw[thick,solid,fill=white] (10cm,-1) circle (.5cm);
  \end{tikzpicture}}
  & 
  \parbox[t]{.3\textwidth}{   \begin{tikzpicture}[scale=.2]
    \draw (-1,0) node[anchor=east]  {$\widetilde{E}_6$};
    \draw[thick] (0 cm,0) -- (8 cm,0);
    \draw[thick] (4 cm,0) -- (4 cm,4);
    \draw[thick,solid,fill=white] (0cm,0) circle (.5cm);
    \draw[thick,solid,fill=white] (2cm,0) circle (.5cm);
    \draw[thick,solid,fill=white] (4cm,0) circle (.5cm);
    \draw[thick,solid,fill=white] (6cm,0) circle (.5cm);
    \draw[thick,solid,fill=white] (8cm,0) circle (.5cm);
    \draw[thick,solid,fill=white] (4cm,2) circle (.5cm);
    \draw[thick,solid,fill=gray] (4cm,4) circle (.5cm);
  \end{tikzpicture}
  
    \begin{tikzpicture}[scale=.2]
    \draw (-1,0) node[anchor=east]  {$\widetilde{E}_7$};
    \draw[thick] (0 cm,0) -- (12 cm,0);
    \draw[thick] (6 cm,0) -- (6 cm,2);
    \draw[thick,solid,fill=white] (0cm,0) circle (.5cm);
    \draw[thick,solid,fill=white] (2cm,0) circle (.5cm);
    \draw[thick,solid,fill=white] (4cm,0) circle (.5cm);
    \draw[thick,solid,fill=white] (6cm,0) circle (.5cm);
    \draw[thick,solid,fill=white] (8cm,0) circle (.5cm);
    \draw[thick,solid,fill=white] (10cm,0) circle (.5cm);
    \draw[thick,solid,fill=gray] (12cm,0) circle (.5cm);
    \draw[thick,solid,fill=white] (6cm,2) circle (.5cm);
  \end{tikzpicture}
  
      \begin{tikzpicture}[scale=.2]
    \draw (-1,0) node[anchor=east]  {$\widetilde{E}_8$};
    \draw[thick] (0 cm,0) -- (14 cm,0);
    \draw[thick] (4 cm,0) -- (4 cm,2);
    \draw[thick,solid,fill=white] (0cm,0) circle (.5cm);
    \draw[thick,solid,fill=white] (2cm,0) circle (.5cm);
    \draw[thick,solid,fill=white] (4cm,0) circle (.5cm);
    \draw[thick,solid,fill=white] (6cm,0) circle (.5cm);
    \draw[thick,solid,fill=white] (8cm,0) circle (.5cm);
    \draw[thick,solid,fill=white] (10cm,0) circle (.5cm);
    \draw[thick,solid,fill=white] (12cm,0) circle (.5cm);
    \draw[thick,solid,fill=gray] (14cm,0) circle (.5cm);
    
    \draw[thick,solid,fill=white] (4cm,2) circle (.5cm);
  \end{tikzpicture}

     \begin{tikzpicture}[scale=.2]
    \draw (-1,0) node[anchor=east]  {$\widetilde{F}_4$};
    \draw[thick] (0 cm,0) -- (2 cm,0);
    \draw[thick] (2 cm,-.2) -- (4 cm,-.2);
    
    \draw[thin] (3.3 cm,.6) -- (2.8 cm,0);
    \draw[thin] (3.3 cm,-.6) -- (2.8 cm,0);
    \draw[thick] (2 cm,.2) -- (4 cm,.2);

    \draw[thick] (4 cm,0) -- (8 cm,0);
    \draw[thick,solid,fill=white] (0cm,0) circle (.5cm);
    \draw[thick,solid,fill=white] (2cm,0) circle (.5cm);
    \draw[thick,solid,fill=white] (4cm,0) circle (.5cm);
    \draw[thick,solid,fill=white] (6cm,0) circle (.5cm);
    \draw[thick,solid,fill=gray] (8cm,0) circle (.5cm);
  \end{tikzpicture}
  
    \begin{tikzpicture}[scale=.2]
    \draw (-1,0) node[anchor=east]  {$\widetilde{G}_2$};
    \draw[thick] (2 cm,0) -- (4 cm,0);
    \draw[thick] (0 cm,-.3) -- (2 cm,-.3);
    \draw[thick] (0 cm,.3) -- (2 cm,.3);
    \draw[thick] (0 cm,0) -- (2 cm,0);
    \draw[thin] (1.3 cm,.7) -- (.8 cm,0);
    \draw[thin] (1.3 cm,-.7) -- (.8 cm,0);
    \draw[thick,solid,fill=white] (0cm,0) circle (.5cm);
    \draw[thick,solid,fill=white] (2cm,0) circle (.5cm);
    \draw[thick,solid,fill=gray] (4cm,0) circle (.5cm);
  \end{tikzpicture}}
  \end{tabular}
  
\end{center}
\caption{The finite and affine Dynkin diagrams (the affine node is marked in gray).}
\label{fig:dynkin}
\end{figure}

\subsection{Weyl Groups}
\label{sec:hyp_and_weyl}
For $\alpha\in\Phi$ and $k\in\Z$, define the affine hyperplane
\[H_{\alpha}^k:=\{x\in V:\langle x,\alpha\rangle=k\}\]
and let 
\[s_{\alpha}^k:x\mapsto x-\frac{2\langle x,\alpha\rangle-k}{\langle\alpha,\alpha\rangle}\]
be the reflection through $H_{\alpha}^k$.  We write $H_{\alpha}$ for the hyperplane $H_{\alpha}^0$ and $s_{\alpha}$ for the reflection $s_{\alpha}^0$.

\medskip

Let $W$ be the group generated by $\{s_{\alpha}:\alpha\in\Phi\}$, called the \defn{Weyl group} of $\Phi$.
It acts on $\Phi$ and is minimally generated by the set $S:=\{s_{\alpha_1},s_{\alpha_2},\ldots,s_{\alpha_n}\}$ of \defn{simple reflections} of $\Phi$.  A \defn{Coxeter element} is a product of the simple reflections in any order, each appearing exactly once.

The \defn{Coxeter arrangement} of $\Phi$ is the central hyperplane arrangement in $V$ given by all the hyperplanes $H_{\alpha}$ for $\alpha\in\Phi$.
The complement $V \setminus \{H_\alpha\}_{\alpha \in \Phi}$ falls apart into connected components, which we call \defn{chambers}. The Weyl group $W$ acts simply transitively on the set of chambers, so we define the \defn{dominant chamber}
\[C:=\{x\in V:\langle x,\alpha\rangle>0\text{ for all }\alpha\in\simp\}\]
and write any chamber as $wC$ for a unique $w\in W$.

\medskip

Let $\wa$ be the group generated by $\{s_{\alpha}^k:\alpha\in\Phi\text,\text{ }k\in\Z\}$, called the affine \defn{Weyl group} of $\Phi$.
It is minimally generated by the set $\widetilde{S}:=S\cup\{s_{\tilde{\alpha}}^1\}$ of \defn{affine simple reflections} of $\Phi$. So we may write any $\waf\in\wa$ as a word in the generators on $\widetilde{S}$. The minimal length of such a word is called the \defn{length} $l(\waf)$ of $\waf$.
It is not hard to see that $\wa$ acts $\Q$. To any $y\in V$, there is an associated translation 
\begin{align*}
 t_{y}:V&\rightarrow V\\
 x&\mapsto x+y.
\end{align*}
If we identify $\Q$ with the corresponding group of translations acting on the affine space $V$, then we may write $\wa=W\ltimes\Q$ as a semidirect product.

The \defn{affine Coxeter arrangement} of $\Phi$ is the affine hyperplane arrangement in $V$ given by all the affine hyperplanes $H_{\alpha}^k$ for $\alpha\in\Phi$ and $k\in\Z$.
Its complement falls apart into connected components, which we call \defn{alcoves}. The affine Weyl group $\wa$ acts simply transitively on the set of alcoves, so we define the (closed) \defn{fundamental alcove} as
\[\ac:=\{x\in V:\langle x,\alpha\rangle\geq 0\text{ for all }\alpha\in\simp\text{ and }\langle x,\tilde{\alpha}\rangle\leq 1\}\]
and write any alcove as $\waf\ac^\circ$ for a unique $\waf\in \wa$, where $\ac^\circ$ is the interior of $\ac$.  We call $\waf$ \defn{dominant} if $\waf \ac^\circ \subseteq C.$

\subsection{Affine Root Systems}
\label{sec:affine_roots}

We may also understand $\wa$ in terms of its action on the set of \defn{affine roots} $\ar$ of $\Phi$. To do this, let $\delta$ be a formal variable and define $\widetilde{V}:=V\oplus\R\delta$.
Define the set of affine roots as
\[\ar:=\{\alpha+k\delta:\alpha\in\Phi\text{ and }k\in\Z\}.\]
If $\waf\in\wa$, write it as $\waf=wt_{\mu}$ for unique $w\in W$ and $\mu\in\Q$ and define
\[\waf(\alpha+k\delta)=w(\alpha)+(k-\langle\mu,\alpha\rangle)\delta.\]
This defines an action of $\wa$ on $\ar$. It imitates the action of $\wa$ on the half-spaces of $V$ defined by the hyperplanes of the affine Coxeter arrangement.
To see this, define the half-space
\[\HH_{\alpha}^k:=\{x\in V:\langle x,\alpha\rangle>-k\}.\]
Then for $\waf\in\wa$ we have $\waf(\HH_{\alpha}^k)=\HH_{\beta}^l$ if and only if $\waf(\alpha+k\delta)=\beta+l\delta$.
Define the set of \defn{positive affine roots} as
\[\ar^+:=\{\alpha+k\delta:\alpha\in\Phi^+\text{ and }k\geq0\}\cup\{\alpha+k\delta:\alpha\in-\Phi^+\text{ and }k>0\},\]
the set of affine roots corresponding to half-spaces that contain $\ac^\circ$. So $\ar$ is the disjoint union of $\ar^+$ and $-\ar^+$.

Define the set of \defn{simple affine roots} as
\[\widetilde{\simp}:=\simp\cup\{-\tilde{\alpha}+\delta\},\]
the set of affine roots corresponding to half-spaces that contain $\ac^\circ$ and share one of its defining inequalities. We will also write $\alpha_0:=-\tilde{\alpha}+\delta$.

For $\waf\in\wa$, we say that $\alpha+k\delta\in\ar^+$ is an \defn{inversion} of $\waf$ if $\waf^{-1}(\alpha+k\delta)\in-\ar^+$, and we write
\begin{align*}
\inv(\waf)&:=\{\alpha+k\delta\in\ar^+:\waf^{-1}(\alpha+k\delta)\in-\ar^+\}\\
&=\ar^+\cap\waf(-\ar^+)
\end{align*}
as the set of inversions of $\waf$. 
\begin{theorem}\label{inv}
 The positive affine root $\alpha+k\delta\in\ar^+$ is an inversion of $\waf$ if and only if the hyperplane $H_{\alpha}^{-k}$ separates $\waf\ac^\circ$ from $\ac^\circ$.
 \begin{proof}
  If $\alpha+k\delta\in\ar^+$ is an inversion of $\waf$, then $\ac^\circ\subseteq\HH_{\alpha}^k$ and $\ac^\circ\nsubseteq \waf^{-1}(\HH_{\alpha}^k)$.
  Thus $\waf\ac^\circ\nsubseteq\HH_{\alpha}^k$ and therefore $H_{\alpha}^{-k}$ separates $\waf\ac^\circ$ from $\ac^\circ$.
  
  Conversely, if $\alpha+k\delta\in\ar^+$ and $H_{\alpha}^{-k}$ separates $\waf\ac^\circ$ from $\ac^\circ$, then $\ac^\circ\subseteq\HH_{\alpha}^k$ and $\waf\ac^\circ\nsubseteq\HH_{\alpha}^k$.
  Therefore $\ac^\circ\nsubseteq \waf^{-1}(\HH_{\alpha}^k)$ and thus $\waf^{-1}(\alpha+k\delta)\in-\ar^+$. So $\alpha+k\delta$ is an inversion of $\waf$.
 \end{proof}

\end{theorem}
Define the \defn{height} of an affine root $\alpha+k\delta$ as $\h(\alpha+k\delta)=\h(\alpha)+kh$.
So $\h(\alpha+k\delta)>0$ if and only if $\alpha+k\delta\in\ar^+$ and $\h(\alpha+k\delta)=1$ if and only if $\alpha+k\delta\in\widetilde{\simp}$.\\
\\
For an integer $l$ with $-h<l<h$, let $\Phi_l$ be the set of roots in $\Phi$ of height $l$. 
Similarly, for any positive integer $b$, let $\ar_b$ be the set of affine roots in $\ar$ of height $b$.
If we write $b=th+r$ with $t,r\in\Z$ and $0\leq r<h$, then
\[\ar_b=\{\alpha+t\delta:\alpha\in \Phi_r\}\cup\{\alpha+(t+1)\delta:\alpha\in \Phi_{r-h}\}.\]


\subsection{Symmetry of the Affine Diagram}
\label{sec:symmetry_of_digram}

Define $\wex:=W\ltimes \cwl$ to be the \defn{extended affine Weyl group} of $\Phi$.
Let 
\[\cycl:=\{\waf\in\wex:\waf\ac=\ac\}.\]
Then $\cycl\cong\wex/\wa\cong\cwl/\Q$ is an abelian group of order $f$, the \defn{index of connection} of $\Phi$. It can be thought of as a group of symmetries of the fundamental alcove $\ac$, or---dually---as a group of symmetries of the affine Dynkin diagram.  The structure of $\cycl$ in simply-laced types is given in~\Cref{fig:cycl}.

\begin{figure}
\[\begin{array}{c|cccccc}
\widetilde{W} & \widetilde{A}_n & \widetilde{D}_{2n} & \widetilde{D}_{2n+1} & \widetilde{E}_6 & \widetilde{E}_7 & \widetilde{E}_8 \\ \hline
\cycl & \mathbb{Z}_{n+1} & \mathbb{Z}_4 & \mathbb{Z}_2 \times \mathbb{Z}_2 & \mathbb{Z}_3 & \mathbb{Z}_2 & \mathbb{Z}_1
\end{array}.\]
\caption{The structures of the abelian groups $\cycl\cong \cwl/\Q$ in affine simply-laced types~\cite{iwahori1965some}, where $\mathbb{Z}_m:=\mathbb{Z}/m\mathbb{Z}$.  Compare with the symmetries of~\Cref{fig:dynkin}.}
\label{fig:cycl}
\end{figure}

\begin{proposition}[{B.~Kostant~\cite[Lemma 3.4.1]{kostant1976macdonald}}]
If $M$ is the Cartan matrix and $c$ is a Coxeter element of $W$, then \[\left|{\sf det}(M)\right|={\sf det}(1-c)=|\cycl|=f.\]
\end{proposition}

We next relate $\cwl$ to a subset of the finite Weyl group $W$.
To do this, we first need a lemma due to Kostant. 
\begin{lemma}[{\cite[Lemma 3.6]{lam2012alcoved}}]
\label{uniquerhoh}
 Every alcove $\waf\ac$ contains exactly one point in $\frac{1}{h}\cwl$. For the fundamental alcove $\ac$, this point is $\rhoh$.
\end{lemma}

\begin{proof}
  We have $\check{\rho}=\sum_{i=1}^n\cw_i\in\cwl$, so $\rhoh\in\frac{1}{h}\cwl$. 
  We also have that $\langle\rhoh,\alpha\rangle=\h(\alpha)/h\in(0,1)$ for all $\alpha\in\Phi^+$.
  Thus $\rhoh$ lies in $\ac^\circ$---in fact, it is the only element in $\ac^\circ\cap\frac{1}{h}\cwl$.
  
  Indeed, suppose that $\nu\in\ac^\circ\cap\frac{1}{h}\cwl$. Then for all $\alpha_i\in\simp$ we have $\langle\nu,\alpha_i\rangle= a_i/h$ for some $a_i\in\Z_{+}$.
  But we also have $\langle\nu,\tilde{\alpha}\rangle=(\sum_{i=1}^na_ic_i)/h<1$, so $a_i=1$ for all $i\in[n]$ and thus $\nu=\rhoh$.
  
  Since $\wa$ acts on $\frac{1}{h}\cwl$, there is exactly one element of $\frac{1}{h}\cwl$ in any alcove $\waf\ac^\circ$.
 \end{proof}

\begin{theorem}
 The point $\rhoh$ is a fixed point of the action of $\cycl$. Any element $\waf$ of $\cycl$ can be written as $\waf=t_{\rhoh}wt_{-\rhoh}$ for a unique $w\in W$.
\end{theorem}

 \begin{proof}
  The extended affine Weyl group acts on $\frac{1}{h}\cwl$. Thus if $\waf\in\cycl$, then $\waf(\rhoh)\in \left(\frac{1}{h}\cwl\right) \cap\ac$, so by Lemma \ref{uniquerhoh} we have $\waf(\rhoh)=\rhoh$.
  Writing $\waf=t_{\mu}w$ for $\mu\in\cwl$ and $w\in W$, we have 
  \[w\left(\rhoh\right)+\mu=\rhoh.\]
  Therefore,
  \[
   \waf=t_{\mu}w=t_{\rhoh-w(\rhoh)}w=t_{\rhoh}t_{-w(\rhoh)}w=t_{\rhoh}wt_{-\rhoh},
  \]
  as required.
 \end{proof}

We now show that in each $\cycl$-orbit of $\check{\Lambda}$, there is exactly one point of $\check{Q}$.  Starting from $\cycl$, let \[b\cycl=\{t_{b\mu}w:t_{\mu}w\in\cycl\}\]  be the analogous group of automorphisms of $b\ac$.

\begin{theorem}
 Let $b$ be a positive integer relatively prime to the index of connection $f$. Then the group $b\cycl$ acts freely on $\cwl$ and every $b\cycl$-orbit contains a unique point in $\Q$.
\label{thm:cycl_action_and_Q}
\end{theorem}
 
 \begin{proof}
  The set of vertices of the fundamental alcove $\Gamma:=\{0\}\cup\{\cw_i:\langle\cw_i,\tilde{\alpha}\rangle=1\}$ is a set of representatives of $\cwl/\Q$.
  Furthermore, we have that $\cycl=\{t_{\mu}w_{\mu}:\mu\in\Gamma\}$, where $w_{\mu}\in W$ for all $\mu\in\Gamma$,
  and $\cycl$ acts simply transitively on $\Gamma$~\cite[Proposition 1.18]{iwahori1965some}.

  Since $b$ is coprime to $f=[\cwl:\Q]$, the map 
  \begin{align*}
   \cwl/\Q&\rightarrow \cwl/\Q\\
   x+\Q&\mapsto bx+\Q
  \end{align*}
  is invertible.  The set $b\Gamma$ is therefore also a set of representatives of $\cwl/\Q$. The group $b\cycl$ acts simply transitively on it.
  We conclude that for any $\mu\in\cwl$ the orbit $(b\cycl)\mu$ is a set of representatives of $\cwl/\Q$. In particular, it is free and contains exactly one point in $\Q$.
 \end{proof}

One can check \emph{case-by-case} that any prime that divides $f$ also divides the Coxeter number $h$.
Thus, the conclusion of the theorem also holds when $b$ is relatively prime to $h$.

\subsection{Strange Identity}
\label{sec:strange_identities}

To conclude this section, we recall the ``strange formula'' of H.~Freudenthal and H.~de Vries. 

\begin{theorem}
Let $\Phi$ be an irreducible crystallographic root system of rank $n$. Let $h$ be the Coxeter number of of $\Phi$, $g$ the dual Coxeter number and recall that $\rho=\frac{1}{2}\sum_{\alpha\in\Phi^+}\alpha$.  Then



 \[\frac{\|\rho\|^2}{2g}=\frac{n(h+1)}{24}.\]
\label{thm:identities}
\end{theorem}

The appearance of the number $24$ in the denominator of the ``strange formula'' will explain its appearance in~\Cref{thm:max_size,thm:expected_size}


\section{Cores}
\label{sec:affine_sym_and_cores}

In this section, we recall the bijection between $a$-cores and the minimal-length coset representatives for the parabolic quotient $\widetilde{\mathfrak{S}}_a / \mathfrak{S}_a$ (\Cref{thm:weak_order_and_cores}).   Using the isomorphism between $\wa/W$ and $\check{Q}$ for $W=\mathfrak{S}_a$, we interpret and generalize cores as points in the coroot lattice $\check{Q}$.

\subsection{The Affine Symmetric Group and Cores}

The \defn{affine symmetric group} has presentation \[\widetilde{\mathfrak{S}}_a := \left \langle s_0,s_1,\ldots,s_{a-1} : (s_i s_{i+1})^3 = e, (s_i s_j)^2 = e \text{ if }|i-j|>1 \right\rangle,\]

where indices will always be taken modulo $a$.  The elements $\waf\widetilde{\mathfrak{S}}_a$ such that $\waf^{-1}\ac^\circ\subseteq C$ are the minimal-length right coset representatives for the parabolic quotient $\widetilde{\mathfrak{S}}_a / \mathfrak{S}_a$.  By abuse of notation, we will associate elements of $\widetilde{\mathfrak{S}}_a / \mathfrak{S}_a$ with their minimal-length right coset representatives.

There is a bijection between $\widetilde{\mathfrak{S}}_a / \mathfrak{S}_a$ and $a$-cores, given as follows.  Label the $(i,j)$th box of the Ferrers diagram of an $a$-core $\lambda$ by its \defn{content} $(j-i) \mod a$.  We define an action $\widetilde{\mathfrak{S}}_a$ on the set of $a$-cores by defining how the simple reflections $s_i$ act.  Given an $a$-core $\lambda$, we define $s_i \lambda$ (for $0 \leq i \leq a-1$) to be the unique $a$-core that differs from $\lambda$ only by boxes with content $i$.  The partial order on $\core(a)$ is given by letting $\lambda$ cover $\mu$ if and only if $\size(\lambda)>\size(\mu)$ and $\lambda = s_i \mu$ for some $i$.

\begin{theorem}[\cite{lascoux2001ordering}]
    There is a poset isomophism $\bijtocore$ between the weak order on the parabolic quotient $\widetilde{\mathfrak{S}}_a / \mathfrak{S}_a$ and the poset on $\core(a)$ defined above.
\label{thm:weak_order_and_cores}
\end{theorem}

Thus, $a$-cores are identified with elements of $\widetilde{\mathfrak{S}}_a / \mathfrak{S}_a$.   Theorem~\ref{thm:weak_order_and_cores} is illustrated in Figure~\ref{fig:weak_order_and_cores} in the case $a=3$.

\begin{figure}[htbp]
\includegraphics[height=3in]{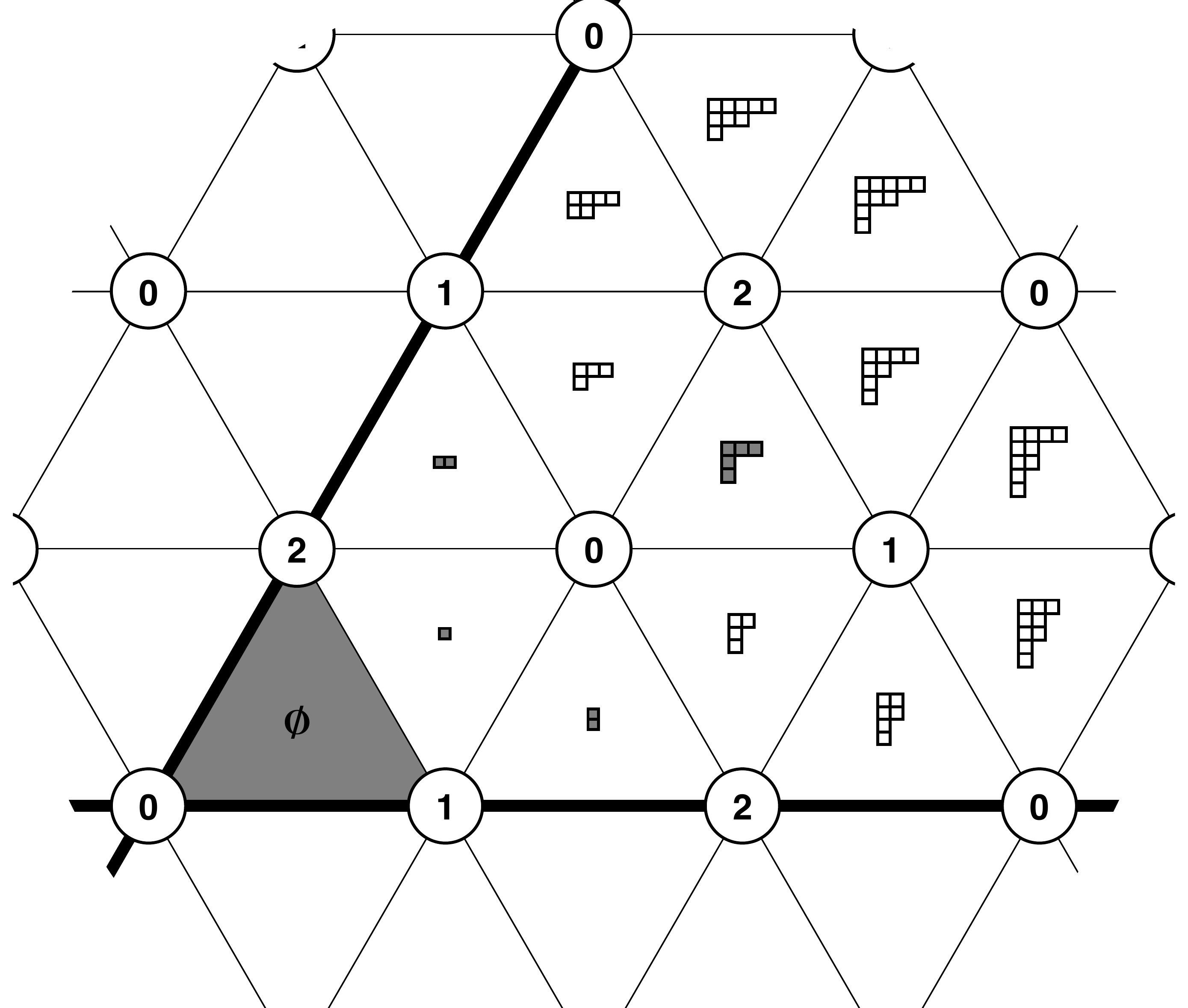}
\caption{The weak order on the minimal-length representatives $\waf\in\widetilde{\mathfrak{S}}_3 / \mathfrak{S}_3$ displayed as the dominant alcoves $\waf^{-1}\ac^\circ$ and the poset of $3$-cores.  The five simultaneous $(3,4)$-cores are the empty core and the four cores shaded gray.}
\label{fig:weak_order_and_cores}
\end{figure}

\subsection{Affine Weyl Groups and Cores}

There are two fundamentally different ways to think of $\wa$.  The first, mentioned in~\Cref{sec:hyp_and_weyl}, is as \[\wa:=W \rtimes \check{Q},\] which we think of as tiling $V$ using bounded copies of $W$ centered at each point of the coroot lattice.  The second way is as \[\wa:=(\wa/W) \rtimes W,\] which may be visualized as replicating a copy of the parabolic quotient $(\wa/W)$ in each region of the Coxeter arrangement of $W$.  These two constructions are related as follows.

\begin{proposition}
There is a canonical bijection \begin{align*}\bijQtoW: (\wa/W) &\to \check{Q} \\ w &\mapsto w(0). \end{align*}
\label{def:bijQtoW}
\end{proposition}

Therefore, the coroot points in $\widetilde{\mathfrak{S}}_a$ are in bijection with the set of $a$-cores.  This suggests that the set of coroot points is the correct generalization of cores to any affine Weyl group $\wa$.  It is natural to write $\core(\wa) := \check{Q}.$

\section{Two Simplices}
\label{sec:dilations_and_cores}

In this section, we recall the definitions of two simplices associated to $\wa$ (\Cref{sec:dilations_and_cores}), and show that they are equivalent up to an explicit rigid transformation (\Cref{thm:wf}). 

\subsection{Dilations of the Fundamental Alcove}

We write 
\[b\ac:=\{x\in V:\langle x,\alpha\rangle\geq 0\text{ for all }\alpha\in\simp\text{ and }\langle x,\tilde{\alpha}\rangle\leq b\}\]
for the $b$-fold dilation of the fundamental alcove, defined for any $b \in \mathbb{R}_{\geq 0}$.  This region is bounded by the hyperplanes
    \[\{H_{\alpha} : \alpha \in \Delta\} \cup \{H_{\alpha_0,b}\}.\]
Its volume is $\rat^n$ times that of the fundamental alcove $\ac$, so it contains $b^n$ alcoves.

\subsection{Sommers Regions} The second simplex, $\Sommers_\Phi(b)$, is defined only for $b$ relatively prime to $h$. In this case write $\rat=th+r$ with $t,r\in\Z_{\geq0}$ and $0<r<h$.
Define the \defn{Sommers region}~\cite{sommers2005b} as
\begin{multline*}
 \Sommers_\Phi(b):=\{x\in V:\langle x,\alpha\rangle\geq-t\text{ for all }\alpha\in \Phi_r\text{ and }
 \langle x,\alpha\rangle\leq t+1\text{ for all }\alpha\in \Phi_{h-r}\}.
\end{multline*}

The significance of the Sommers region is as follows. Define $\wa^b$ as the set of $\waf\in\wa$ that have no inversions of height $b$. So $\wa^b=\{\waf\in\wa:\waf^{-1}(\ar_b)\subseteq\ar^+\}$.
By~\Cref{inv}, we have that $\waf\in\wa^b$ if and only if none of the affine hyperplanes in 
\[\{H_{\alpha}^{-t}:\alpha\in \Phi_r\}\cup\{H_{\alpha}^{-t-1}:\alpha\in \Phi_{r-h}\}=\{H_{\alpha}^{-t}:\alpha\in \Phi_r\}\cup\{H_{\alpha}^{t+1}:\alpha\in \Phi_{h-r}\}\]
separate $\waf\ac^\circ$ from $\ac^\circ$. So $\waf\in\wa^b$ if and only if $\waf\ac^\circ\subseteq\Sommers_{\Phi}(b)$.

\subsection{From $\Sommers_{\Phi}(\rat)$ to $\rat\ac$}
\label{sec:somtoac}

It turns out that $b\ac$ and $\Sommers_\Phi(b)$ are equivalent up to a rigid transformation, which may be realized as an element $\wf \in \wa$.  

\begin{theorem}\label{rhoh}
 For $\rat$ relatively prime to $h$, there exists a unique element $\waf_{\rat}=t_{\mu}w\in\wa$ with
 \[\rat\rhoh=\waf_{\rat}\left(\rhoh\right).\]
 \end{theorem}

 \begin{proof}
 For all $\alpha\in\Phi^+$, we have that 
  \[\left\langle \rat\rhoh,\alpha\right\rangle=\rat\frac{\h(\alpha)}{h}\notin\Z,\]
  since $\rat$ is relatively prime to $h$ and $h$ does not divide $\h(\alpha)$. Thus $\rat\rhoh$ lies on no hyperplane of the affine Coxeter arrangement, so it is contained in some alcove $\wf\ac^{\circ}$. Since $\rat\rhoh\in\frac{1}{h}\cwl$ we have that $\rat\rhoh=\wf(\rhoh)$ by Lemma \ref{uniquerhoh}.
\end{proof}
We were unable to find the following result---which explicitly identifies the rigid transformation that sends $\Sommers_\Phi(b)$ to $\rat\ac$ as an element of $\wa$---in the literature, although it is probably well-known to the experts.\footnote{The closest result to~\Cref{thm:wf} we were able to find was an existence statement in~\cite[Proof of Theorem 5.7]{sommers2005b}, which relies upon Lemma 2.2 and the end of Section 2.3 in~\cite{fan1996euler}.  See also~\cite[Theorem 4.2]{athanasiadis2005refinement}.}

\begin{theorem}
 The affine Weyl group element $\waf_{\rat}=t_{\mu}w$ maps $\Sommers_{\Phi}(\rat)$ bijectively to $\rat\ac$.
 \label{thm:wf}
\end{theorem}

 \begin{proof}
  We calculate that
  \begin{align*}
   \frac{\h(\alpha)}{h}=\left\langle\rhoh,\alpha\right\rangle=\left\langle w\left(\rhoh\right),w(\alpha)\right\rangle=\left\langle \rat\rhoh-\mu,w(\alpha)\right\rangle=\rat\frac{\h(w(\alpha))}{h}-\langle\mu,w(\alpha)\rangle
  \end{align*}
  
  Thus $\h(\alpha)=\rat\h(w(\alpha))-h\langle\mu,w(\alpha)\rangle$. Again write $\rat=th+r$ with $t,r\in\Z$ and $0<r<h$.
  So reducing modulo $h$ we get $\h(\alpha)\equiv r\h(w(\alpha))\mod h$.
  Thus $\h(\alpha)\equiv r\mod h$ if and only if $\h(w(\alpha))\equiv 1\mod h$. So 
  
  \[w(\Phi_r\cup \Phi_{r-h})=\Phi_1\cup \Phi_{1-h}=\simp\cup\{-\tilde{\alpha}\}.\]
  
  For $\alpha\in\simp$, we have
  \begin{align*}
   \frac{\h(w^{-1}(\alpha))}{h}=\rat\frac{\h(\alpha)}{h}-\langle\mu,\alpha\rangle=\frac{\rat}{h}-\langle\mu,\alpha\rangle.
  \end{align*}
  
  Now $\h(w^{-1}(\alpha))$ equals either $r$ or $r-h$, so $\langle\mu,\alpha\rangle=t$ if $w^{-1}(\alpha)\in\Phi^+$ and $\langle\mu,\alpha\rangle=t+1$ if $w^{-1}(\alpha)\in-\Phi^+$.
  Comparing with \cite[Section 2.3]{fan1996euler} ($w=w'$, $\mu=\nu$) gives the result.
 \end{proof}

We stress that this bijection sends the set of coweight lattice points in $\Sommers_\Phi(b)$ to the set of coweight lattice points in $\rat\ac$, and similarly for coroot lattice points.

Now~\cite[Theorem 8.2]{thiel2015anderson} implies that this element is unique.
\begin{theorem}
 $\waf_{\rat}$ is the unique $\waf\in\wa$ with $\waf(\Sommers_\Phi(b))=\rat\ac$.
\end{theorem}

\section{Simultaneous Cores}
\label{sec:simultanteous_cores}

  We generalize simultaneous cores to the points $\Sommers_\Phi(b) \cap \check{Q}$ (\Cref{def:w_cores}).  In~\Cref{sec:enumeration}, we recall M.~Haiman's result on the number of such simultaneous cores (\Cref{thm:haiman_number}).

\subsection{Definition}

We now recall how to identify simultaneous $(a,b)$-cores using the bijection of~\Cref{def:bijQtoW}.  Let $\bijqtocore := \bijtocore \circ \bijQtoW^{-1}: \check{Q} \to \core(a)$.

\begin{proposition}[\cite{garvan1990cranks,johnson2015lattice}]
    For $\gcd(a,b)=1$, \[\core(a,b) = \{ \bijqtocore(\lambda) : \lambda \in \Sommers_{\Phi}(b) \cap \check{Q}\},\] where $\Phi$ is a root system for $\mathfrak{S}_a$.
\end{proposition}

We conclude that the coroot points in $\Sommers_\Phi(b)$ generalize the set of simultaneous $(a,b)$-cores to all $\wa$.  We emphasize this with the following definition.

\begin{definition}
    For $b$ relatively prime to $h$, we write \[\core(\wa,b) := \Sommers_\Phi(b) \cap \check{Q}.\]
\label{def:w_cores}
\end{definition}

\subsection{Enumeration:~\Cref{thm:haiman_number}}
\label{sec:enumeration}

We could count $\left|\core(\wa,b)\right|$ using the relationship between the lattices $\check{\Lambda}$ and $\check{Q}$ given in~\Cref{thm:cycl_action_and_Q}, and between the simplices $\Sommers_\Phi(b)$ and $b\ac$ in~\Cref{thm:wf}.  If we write a coweight lattice point in the coweight basis as $(x_0,x_1,x_2,\cdots,x_n)$, then the coweight points $b\ac \cap \check{\Lambda}$ are exactly the positive integral solutions to the linear equation
\begin{equation}
\sum_{i=0}^n c_i x_i = b, \hspace{5em} 0 \leq x_i \in \mathbb{Z}.
\label{eq:weights_inside_ba}
\end{equation}

 where $\tilde{\alpha}=\sum_{i=1}^n c_i \alpha_i$ and we take $c_0:=1$.  This last reformulation is easily counted~\cite{suter1998number}: simply expand the generating function \[\prod_{i=0}^n \frac{1}{1-q^{c_i}}=\sum_{b=0}^\infty \left|b\ac \cap \check{\Lambda} \right|q^b.\]

\begin{example}
For an affine root system of type $\mathfrak{S}_a$, we have that $c_i=1$.  It is clear that there are $\binom{a+b-1}{b}$ coweight lattice points that are solutions to~\Cref{eq:weights_inside_ba}.  Dividing by the index of connection, we obtain the corresponding number of coroot lattice points $\left| \core(a,b)\right| = \frac{1}{a} \binom{a+b-1}{b}$.
\end{example}

To do this sort of type-by-type analysis, however, is to overlook M.~Haiman's beautiful \emph{nearly} uniform proof\footnote{M.~Haiman's proof uses the fact that any prime that divides a coefficient of the highest root also divides the Coxeter number $h$, for which we don't know a uniform proof.}, which combines P\'{o}lya theory, the Shephard-Todd formula, and---perhaps most surprisingly---Dirichlet's theorem that any infinite arithmetic sequence of positive integers contains an infinite number of primes~\cite{haiman1994conjectures}.  See also~\cite{sommers2005b}.

{
\renewcommand{\thetheorem}{\ref{thm:haiman_number}}
\begin{theorem}[{Number of $(\wa,b)$-cores; M.~Haiman~\cite{haiman1994conjectures}}]
$ $\newline
	For $\gcd(h,b)=1$, \[\left|\core(\wa,b)\right| = \frac{1}{|W|} \prod_{i=1}^n (b+e_i).\]
\end{theorem}
\addtocounter{theorem}{-1}
}

This number has since become known as the \defn{rational Catalan number} associated to $W$ and $b$ (see, for example,~\cite{armstrong2013rational}).

\section{The Statistic $\size$}
\label{sec:statistics}


In this section, we interpret the statistic $\size$, which counts the number of boxes in the Ferrers diagram of an $a$-core, in the language of root systems.  We then extend this statistic to any affine Weyl group.

\subsection{$\size$ on Elements of $\wa$}
\begin{definition}
For $\waf \in \wa$, define
\[\size(\waf):=\sum_{\alpha + k \delta \in \inv(\waf^{-1})} k.\]
\label{def:size_on_elements}
\end{definition}

\begin{example}
    The core $\lambda = \raisebox{0.5\height}{\tiny\ydiagram{3,1,1}\normalsize}$ in $\widetilde{A}_2$ has 5 boxes.  The corresponding affine element $\waf = s_1s_2s_1s_0$---illustrated in~\Cref{fig:weak_order_and_cores}---has $\size(\waf)=5$ because the inversions of $\waf^{-1}$ are the affine roots $-\amax+1\cdot\delta, -\alpha_1+1\cdot\delta,-\alpha_2+1\cdot\delta,-\amax+2\cdot\delta,$ and $5=1+1+1+2.$
\label{ex:heights}
\end{example}

\begin{remark}
We specify here that the statistic $\size$ on elements of $\check{Q}$ in type $\widetilde{C}_n$ is \emph{not} equal to the number of boxes of the corresponding self-conjugate core (a model studied, for example, in~\cite{baldwin2006self,ford2009self,hanusa2013number,chen2014average,alpoge2014self}).  For example, one can compute that the element $\waf=s_0s_1s_0s_1s_2s_1s_0\in \widetilde{C}_2$ has $\size(\waf) = 11$, but that $\waf$ corresponds to the self-conjugate core \raisebox{1.5\height}{\tiny\ydiagram{6,3,3,1,1,1}\normalsize}, which has 15 boxes.

There is a simple way to read off $\size$ in $\widetilde{C}_n$ on a self-conjugate core, which we state here without proof: weight by 2 those boxes $(i,j)$ such that $i < j$ and $j-i = 0 \mod n$, by 1 the remaining boxes $(i,j)$ such that $i \leq j$, and by 0 all other boxes.  Then $\size$ is given by the sum of the weights of the boxes.  For $\waf$ as above, we have the weighting \raisebox{1.5\height}{\tiny\begin{ytableau}1 &1&2&1&2&1 \\ 0&1&1\\0&0&1 \\ 0\\0\\0\end{ytableau}\normalsize}.  The sum of these weights is the desired $11=\size(\waf)$.
\label{rem:size_not_number}
\end{remark}

The statistic $\size$ is preserved under the bijection between minimal coset representatives of $\widetilde{\mathfrak{S}}_a/\mathfrak{S}_a$ and $a$-cores.

\begin{proposition}
The bijection \begin{align*}
  \bijtocore:\widetilde{\mathfrak{S}}_a/\mathfrak{S}_a&\rightarrow\core(a)\\
  \waf\mapsto \waf\cdot\emptyset
 \end{align*}
 preserves $\size$.
\label{prop:size_in_type_a}
\end{proposition}

 \begin{proof}
  By~\cite[Proposition 8.2]{fishel2010bijection}, if $\kappa=\bijtocore(\waf)$ is an $a$-core, $k>0$ and $s_i\in\widetilde{\simp}$ is such that $l(s_i\waf)>l(\waf)$,
  then $s_i\kappa$ has $k$ more boxes than $\kappa$ if and only if the unique hyperplane that separates $\waf^{-1}s_i\ac^\circ$ from $\waf^{-1}\ac^\circ$ is of the form $H_{\alpha}^k$ for some $\alpha\in\Phi^+$.
  Thus by induction on $l(\waf)$ the size of the core $\waf\cdot\emptyset$ is
  \[\sum_{\substack{H_{\alpha}^k\text{ separates $\waf^{-1}\ac^\circ$ from $\ac^\circ$}\\\alpha\in\Phi^+}}k,\]
  which equals
  \[\sum_{\alpha+k\delta\in\mathsf{Inv}(\waf^{-1})}k\]
  by Theorem \ref{inv}. Therefore $\size(\waf\cdot\emptyset)=\size(\waf)$.
 \end{proof}

The statistic $\size$ of~\Cref{def:size_on_elements} therefore generalizes the statistic $\size$ on $a$-cores to all elements of an affine Weyl group.

\subsection{$\size$ as a Quadratic Form}

We can also view $\size$ as a statistic on the coroot lattice $\Q$.
To do this, we follow~\cite{macdonald1971affine}, although it is possible to argue directly using the ``strange identities'' of~\Cref{sec:strange_identities}.
I.~G.~Macdonald views the affine root $\alpha+k\delta\in\ar$ as the affine linear functional
\begin{align*}
 \alpha+k\delta:V&\rightarrow\R\\
 x&\mapsto\langle x,\alpha\rangle+k.
\end{align*}
For $\waf\in\wa$ he defines $s(\waf):=\sum_{\alpha+k\delta\in\mathsf{Inv}(\waf)}\alpha+k\delta$.
Furthermore, I.~G.~Macdonald introduces a quadratic form $\Psi$ as
\[\Psi(x):=\frac{g}{2}\left\|x-\rhog\right\|^2.\]

\begin{lemma}[Proposition 7.5 in Macdonald]
 We have 
 \[s(\waf)=\Psi\circ \waf^{-1}-\Psi\]
 for all $\waf\in\wa$.
 \label{lem:quad_form_psi}
\end{lemma}

Using this, we can calculate $\size(\waf)$ as follows:
\begin{align*}
 \size(\waf)&=\sum_{\alpha+k\delta\in\mathsf{Inv}(\waf^{-1})}k\\
 &=\left(\sum_{\alpha+k\delta\in\mathsf{Inv}(\waf^{-1})}\alpha+k\delta\right)(0)\\
 &=s(\waf^{-1})(0)\\
 &=(\Psi\circ \waf-\Psi)(0)\\
 &=\frac{g}{2}\left\|\waf(0)-\rhog\right\|^2-\frac{g}{2}\left\|\rhog\right\|^2\\
 &=\frac{g}{2}\left\|\waf(0)\right\|^2-\langle\waf(0),\rho\rangle.
\end{align*}

This suggests the following definition of $\size$ as a quadratic form on $V$. 

\begin{definition}
For any $x\in V$, define
\begin{align*}
\size(x)&:=\frac{g}{2}\left\|x-\rhog\right\|^2-\frac{g}{2}\left\|\rhog\right\|^2\\
&=\frac{g}{2}\left\|x\right\|^2-\langle x,\rho\rangle.
\end{align*}
\label{def:size_and_quadr}
\end{definition}
Note that by the ``strange'' formula (\Cref{thm:identities}), we have $\frac{g}{2}\left\|\rhog\right\|^2=\frac{n(h+1)}{24}$.

The computation directly after~\Cref{lem:quad_form_psi} shows that we have the following analogue of~\Cref{prop:size_in_type_a}.
\begin{corollary}
 The bijection 
 \begin{align*}
  \bijQtoW:\wa/W&\rightarrow\Q\\
  \waf&\mapsto \waf(0)
 \end{align*}
 preserves $\size$.
\end{corollary}

\subsection{Translating $\size$ from $\Sommers_\Phi(\rat)$ to $\rat\ac$}

To resolve the first obstacle raised in~\Cref{sec:technical_difficulties}---that the orientation of the Sommers region changes with the residue class of $b \mod h$---we wish to transfer the $\size$ statistic from the Sommers region $\Sommers_\Phi(b)$ to the dilated fundamental alcove $\rat\ac$.
Using the results of~\Cref{sec:somtoac}, we define \[\zise(x):=\size(w_\rat^{-1}(x)) \text{ for all } x\in V,\] so that the bijection $w_\rat$ sends $\size$ on $\Sommers_\Phi(b)$ to $\zise$ on $\rat\ac$.

\begin{corollary}
For $b$ coprime to $h$ we have
\[\Big\{\size(\lambda) : \lambda \in \core(\wa,b)\Big\} = \Big\{\sizeshift(\lambda) : \lambda \in \rat\ac \cap \check{Q}\Big\}.\]
\label{cor:size_on_ba}
\end{corollary}

\subsection{The simply-laced condition}
\label{sec:simply_laced}

It will be useful to define the $W$-invariant quadratic form $\quadr$ on $V$ by
\begin{align}\label{thm:explicit_quadr}
\quadr(x)&:=\frac{g}{2}\|x\|^2-\frac{g}{2}\left\|\rhog\right\|^2\\
&=\frac{g}{2}\|x\|^2-\frac{n(h+1)}{24}
\end{align}
so that $\size(x)=\quadr(x-\rhog)$ for all $x\in V$.

\medskip

Using the quadratic form $\quadr$, we find a considerable simplification of $\zise$ in the case where $\Phi$ is simply laced.

\begin{theorem}
For $\wa$ simply laced,
\[\zise(x) = \frac{h}{2}\left\|x-\rat\rhoh\right\|^2-\frac{n(h+1)}{24}\]
\label{thm:simple_zise}
\end{theorem}

\begin{proof}
From~\Cref{rhoh,thm:wf} we have that $\wf=t_{\mu}w$ for $\mu\in\Q$ and $w\in W$ such that $\rat\rhoh=w(\rhoh)+\mu$.
We calculate 
\begin{align*}
 \wf&=t_{\mu}w=t_{\rat\rhoh-w(\rhoh)}w=t_{\rat\rhoh}t_{-w(\rhoh)}w=t_{\rat\rhoh}wt_{-\rhoh}.
\end{align*}
We conclude that
\begin{align*}
 \zise(x)&=\size\left(\wf^{-1}(x)\right)\\
 &=\size\left((t_{\rat\rhoh}wt_{-\rhoh})^{-1}(x)\right)\\
 &=\size\left(w^{-1}\left(x-\rat\rhoh\right)+\rhoh\right)\\
 &=\quadr\left(w^{-1}\left(x-\rat\rhoh\right)+\rhoh-\rhog\right).
\end{align*}

  By assumption $\wa$ is simply-laced, so that $\rho=\check{\rho}$ and $g=h$.  Then
\begin{align*}
 \zise(x)&=\quadr\left(w^{-1}\left(x-\rat\rhoh\right)\right)=\quadr\left(x-\rat\rhoh\right),
\end{align*}
since $\quadr$ is $W$-invariant.
\end{proof}

\begin{remark}
The simplification of $\zise$ in~\Cref{thm:simple_zise} is the origin of the simply-laced condition in Theorems~\ref{thm:macdonald_number},~\ref{thm:max_size},~\ref{thm:expected_size}, and~\ref{thm:variance}.
\label{rem:simplification}
\end{remark}

\subsection{$\size$ and Symmetry}

We show that the action of $b\cycl$ on $b\ac$ preserves $\zise$. In particular, this dispenses with the second obstacle from~\Cref{sec:technical_difficulties}, allowing us to study the coweights $b\ac \cap \check{\Lambda}$ instead of the coroots $b\ac\cap \check{Q}$.

\begin{lemma}
The group $b\cycl$ preserves the statistic $\zise$ on $b\ac$.
\label{thm:cycl_action_and_quad}
\end{lemma}

\begin{proof}
For $\waf=t_{b\rhoh}wt_{-b\rhoh}\in b\cycl$ and $x\in b\ac$ we have
\[
 \zise(\waf(x)) = \quadr\left(\waf(x)-b\rhoh\right)=\quadr\left(w\left(x-b\rhoh\right)\right)=\quadr\left(x-b\rhoh\right)=\zise(x).
\]
\end{proof}

\subsection{\Cref{thm:max_size}: Maximum $\size$ of a $(\widetilde{W},b)$-core}

In this section we restate and prove~\Cref{thm:max_size}, establish a connection between $\wf$ and rational $(h,b)$-Dyck paths (\Cref{thm:w_b_inversion_set}), and conjecture that $\wf$ is the largest element in weak order among all dominant elements corresponding to $\core(\widetilde{W},b)$ (\Cref{conj:w_b_is_maximal_in_weak}).

{
\renewcommand{\thetheorem}{\ref{thm:max_size}}
\begin{theorem}[{Maximum $\size$ of a $(\widetilde{W},b)$-core}]
$ $\newline
	For $\widetilde{W}$ a simply-laced affine Weyl group with $\gcd(h,b)=1$, \[\max_{\lambda \in \Sommers_\Phi(b) \cap \check{Q}}(\size(\lambda)) = \frac{n(b^2-1)(h+1)}{24}.\]  This maximum is attained by a unique point $\lambda \in \Sommers_{\widetilde{W}}(b)$.
\end{theorem}
\addtocounter{theorem}{-1}
}

\begin{proof}
We claim that the maximum is obtained at $\lambda = \wf^{-1}(0)$.  First note that since $\wf$ maps $\Sommers_{\Phi}(\rat)$ bijectively to $b\ac$, $\wf^{-1}(0)$ is indeed in $\core(\wa,\rat)=\Sommers_{\Phi}(\rat)\cap\Q$.
  Since $\wf$ maps $\size$ to $\zise$ (\Cref{cor:size_on_ba}), we will show the equivalent statement that $0$ is the unique element of $\rat\ac\cap\Q$ of maximum $\zise$.
  We have that
  \begin{align*}
   \zise(x)&=\quadr\left(x-\rat\rhoh\right)\\
   &=\frac{h}{2}\left\|x-\rat\rhoh\right\|^2-\frac{n(h+1)}{24}
  \end{align*}
  is a strictly convex function in $x$, and so it can only be maximized at a vertex of the convex polytope $\rat\ac$. We will show that among all the vertices of $\rat\ac$, the vertex $0$ has maximal $\zise$.
  Together with the fact that $0$ is the only vertex of $\rat\ac$ that is in the coroot lattice $\Q$ this implies the result.\\
  \\
  Let $x_0,x_1,\ldots,x_n$ be the vertices of $\ac$, where $x_0=0$ and $x_i$ is the vertex with $\langle x_i,\alpha_i\rangle>0$ for $i\in[n]$.
  So $\rat x_0,\rat x_1,\ldots,\rat x_n$ are the vertices of $\rat\ac$.
  We wish to show that $\|\rat x_i-\rat\rhoh\|^2$ is maximal for $i=0$. For this it is sufficient to show that $\|x_i-\rhoh\|^2$ is maximal for $i=0$.\\
  \\
  Define $\alpha_0=-\amax$ and for any $i=0,1,\ldots,n$ let $\Phi_{i}$ be the root system whose set of simple roots is $\{\alpha_0,\alpha_1,\ldots,\alpha_n\}\backslash\{\alpha_i\}$.
  Define $\rho_i=\frac{1}{2}\sum_{\alpha\in\Phi_i^+}\alpha$.
  Then by~\cite[Proposition 7.3]{macdonald1971affine} (using that $\rho=\check{\rho}$ and $g=h$) we have $x_i-\rhoh=-\frac{\rho_i}{h}$ for all $i\in\{0,1,\ldots,n\}$.
  So we just need to check \emph{case-by-case} that $\|\rho_i\|^2$ is maximized when $i=0$. This is easily accomplished.\\
\\
We explictly compute this maximum: 
\begin{align*}
\size(\wf^{-1}(0))&=\zise(0)=Q\left(-\frac{b\rho}{h}\right)=\frac{g}{2}\left\|\rat\rhoh\right\|^2-\frac{g}{2}\left\|\rhoh\right\|^2\\
&=(b^2-1)\frac{g}{2}\left\| \rhoh \right\|^2
=(b^2-1)\frac{g^2n(h+1)}{24h^2}\\
&=\frac{n(b^2-1)(h+1)}{24}.
\end{align*}
\end{proof}

\medskip

We now characterize the inversion set of the affine element $\wf$.


\begin{theorem}
For $\wa$ a simply-laced affine Weyl group $b$ relatively prime to $h$ and $\wf$ as in~\Cref{thm:wf},
\[\inv(\wf) = \left\{-\alpha+k\delta:\alpha\in\Phi^+,0<k<\frac{b}{h}\h(\alpha)\right\}.\]
\label{thm:w_b_inversion_set}
\end{theorem}

\begin{proof}
From~\Cref{rhoh}, we know that $\waf_{\rat}\left(\rhoh\right)=\rat\rhoh.$  Then the result follows from calculating $\left\langle \rat \rhoh, \alpha \right\rangle=\frac{b}{h}\h(\alpha)$ for $\alpha \in \Phi^+$.
\end{proof}

Let $\gcd(h,b)=1$.  If we draw a line of rational slope in $\mathbb{R}^2$ from the point $(0,0)$ to the point $(h,b)$, then for $1 \leq i \leq b-1$, the number of boxes with $y$-coordinate equal to $i$ is given by the sequence \[
\left\{\left\lfloor \frac{i h}{b} \right\rfloor\right\}_{i=1}^{b-1}.\]  By~\Cref{thm:w_b_inversion_set}, this sequence characterizes the inversion set of $\wf$.

Summing the inversions of~\Cref{thm:w_b_inversion_set} rank-by-rank gives the following corollary.

\begin{corollary}
	For $\wa$ a simply-laced affine Weyl group with $\gcd(h,b)=1$,
\[		\sum_{i=1}^{b-1} (b-i) \sum_{j=1}^{\lfloor \frac{i h}{b} \rfloor} \left|\Phi_{h-j} \right| = \frac{n(b^2 - 1)(h + 1)}{24},\]
where $\Phi_{\geq i}$ is the set of positive roots of height greater than or equal to $h-i$.
\label{cor:max_size_as_sum}
\end{corollary}

Since we can easily write down explicit formulas for $\left|\Phi_{h-j} \right|$, we obtain apparently nontrivial identities involving the floor function.  For example, in type $\widetilde{A}_n$ with $\gcd(n+1,b)=1$, we obtain the equality 

\[		\sum_{i=1}^{b-1} \frac{b-i}{2} \left\lfloor \frac{i (n+1)}{b} \right \rfloor \left(1+ \left\lfloor \frac{i (n+1)}{b} 
\right\rfloor\right) = \frac{n(b^2 - 1)(n + 2)}{24}.\]

In type $\widetilde{D}_n$ with $\gcd(2n-2,b)=1$, we have

\[ \sum_{i=1}^{b-1} (b-i) \left( \sum_{j=1}^{\min\left(\lfloor \frac{i (2n-2)}{b} \rfloor , n-2\right)} \left\lfloor \frac{j+1}{2} \right\rfloor + \sum_{j=n-2}^{\left\lfloor \frac{i (2n-2)}{b} \right\rfloor-1} \left\lceil \frac{j+3}{2} \right\rceil \right) = \frac{n(b^2-1)(2n-1)}{24}.\]

We challenge the reader to prove these equalities directly!

\medskip

Although~\Cref{thm:max_size} proves that $\size(\wf^{-1}(0))$ is the maximum that the statistic $\size$ can take on $\core(\wa,b)$, we believe that the inversion set of $w_b$---specified in~\Cref{thm:w_b_inversion_set}---contains the inversion sets of all other affine elements corresponding to elements of $\core(\wa,b)$.  This conjecture generalizes J.~Vandehey's result that the largest $(a,b)$-core contains all other $(a,b)$-cores as subdiagrams (see~\cite{vandehey2008general,fayers2011t}).

\begin{conjecture}
The element $\wf$ is maximal in the weak order on $\wa/W$ among all dominant elements $\{ \waf \in \wa/W : \waf^{-1}(0) \in \Sommers_\Phi(b)\}$.
\label{conj:w_b_is_maximal_in_weak}
\end{conjecture}

\section{Calculations}
\label{sec:proofs}

We begin with a review of \emph{weighted} Ehrhart theory, which extends the quasipolynomiality and reciprocity theorems of usual Ehrhart theory to weighted sums over lattice points in a rational polytope (\Cref{sec:weighted_ehrhart}).  In~\Cref{sec:outline_of_calculations}, we outline the calculations we will perform, pulling together the theory from the previous parts of the paper.  We work out these calculations by hand in~\Cref{sec:type_a_variance,sec:type_d_expected} to find the variance in type $\widetilde{A}_n$ and the expected value in type $\widetilde{D}_n$.  In~\Cref{sec:automation}, we detail our methodology for automating these computations, which allows us to compute the third moment in type $\widetilde{A}_n$ and the variance in type $\widetilde{D}_n$.  In~\Cref{sec:e_series}, we explain how to verify~\Cref{thm:expected_size} in types $\widetilde{E}_6$, $\widetilde{E}_7$, and $\widetilde{E}_8$ using the freely available program $\mathsf{Normaliz}$  \cite{bruns2001normaliz,bruns2010normaliz}.

\subsection{Weighted Ehrhart Theory}
\label{sec:weighted_ehrhart}

    Fix $\mathcal{P}$ a $n$-dimensional rational convex polytope in a lattice $L$ (with generators a basis of $\mathbb{R}^n$), and let $h:\mathbb{R}^n \to \mathbb{R}$ be a polynomial of degree $r$.  The \defn{weighted lattice-point enumerator} for the $b$th dilate of $\mathcal{P}$ is \[\mathcal{P}^L_h(b):=\sum_{x \in b\mathcal{P} \cap L} h(x).\]  It turns out that $\mathcal{P}^L_h(b)$ is not only a quasipolynomial, but also satifies a reciprocity relation.\footnote{Even though these results are well-known to the experts, it is difficult to find explicit statements in the literature that apply at once to \emph{rational} polytopes and \emph{weighted} lattice-point enumerators; see the remark in~\cite[Section 1.2.2]{johnson2015lattice}.}

\begin{theorem}[{\cite{barvinok2006computing,berline2007local,baldoni2012computation},~\cite[Theorem 4.6]{ardila2014double}}]
	\label{thm:generized_ehrhart}
	For $\mathcal{P}$, $L$, and $h$ as above,
\begin{enumerate}
\item $\mathcal{P}^L_h(b)$ is a quasipolynomial in $b$ of degree $n+r$.  Its period divides the least common multiple of the denominators of the coordinates (in the generators of $L$) of the vertices of $\mathcal{P}$.
\item If $\mathcal{P}^o$ is the interior of $\mathcal{P}$, then \[\mathcal{P}^L_h(-b) = (-1)^n (\mathcal{P}^o)^L_h(b).\]
\end{enumerate}
\end{theorem}

When $h(x)=1$, we have $\mathcal{P}^L_h(b)=\left | b\mathcal{P} \cap L\right|$ and we therefore recover the well-known theorems of E.~Ehrhart and I.~G.~MacDonald.  We refer the reader to~\cite{beck2007computing} for further information and definitions.  We will use the notation $\mathcal{P}^L_h(b)_i$ to refer to the $i$th component of the quasipolynomial $\mathcal{P}^L_h(b)$.

\subsection{Outline of Calculations}  
\label{sec:outline_of_calculations}

Drawing heavily from P.~Johnson's proof in~\cite{johnson2015lattice} of the expected size of a simultaneous core in type $\widetilde{A}_n$, we outline our methodology.

In~\Cref{thm:explicit_quadr}, we proved that the statistic $\size$ is a quadratic form.  \Cref{cor:size_on_ba} then transfers $\size$ to a statistic $\sizeshift$ on $b\ac \cap \check{\Lambda}$, resolving the first obstacle outlined in~\Cref{sec:technical_difficulties} (that our desired region was changing orientation as we changed the dilation factor). 

The following proposition identifies $\ac$ as a rational polytope, so that the theory in~\Cref{sec:weighted_ehrhart} resolves the second obstacle in~\Cref{sec:technical_difficulties} (that $\Sommers_\Phi(b)$ is not an integer polytope).

\begin{proposition}
	The polytope $\ac$ is rational in the coweight lattice $\check{\Lambda}$.
\label{prop:barational}
\end{proposition}

\begin{proof}
	The vertices of $\ac$ are given by the set \[\Gamma:=\{0\}\cup\{\frac{\cw_i}{c_i} : 1 \leq i \leq n\},\] where we remind the reader that the $\cw_i$ are the fundamental coweights and $\tilde{\alpha}$ is the highest root.  These vertices are rational in the lattice $\check{\Lambda}$.
\end{proof}

Part (1) of \Cref{thm:generized_ehrhart} now allows us to conclude that the weighted lattice-point enumerator \[\ac^{\check{\Lambda}}_{\sizeshift^k}(b) = \sum_{\mu \in b \ac \cap \check{\Lambda}} \sizeshift^k(x) = \sum_{\mu \in \Sommers_\Phi(b) \cap \check{\Lambda}} \size^k(x)\] is a quasipolynomial of of degree $n+2k$ of period $m(\widetilde{W}):=\lcm(c_1,\ldots,c_n)$, where the $c_i$ are the coefficients of the simple roots in the highest root---by~\Cref{prop:barational}, the $c_i$ are the denominators of the coordinates of the vertices of $\ac$.  One can check that $m(\widetilde{A}_n)=1$, $m(\widetilde{D}_{n})=2$, $m(\widetilde{E}_6)=6$, $m(\widetilde{E}_7)=12$, and $m(\widetilde{E}_8)=60$.  Outside of type $\widetilde{A}_n$, this imposes an additional constraint: we want to pick $b$ that is in the correct residue class modulo $m(\widetilde{W})$.  On the other hand, what was only \textit{a priori} a quasipolynomial actually collapses to a \emph{polynomial} in type $\widetilde{A}_n$.

To deduce the desired formula for \[\ac^{\check{Q}}_{\sizeshift^k}(b) = \sum_{\lambda \in b \ac \cap \check{Q}} \sizeshift^k(x) = \sum_{\lambda \in \Sommers_\Phi(b) \cap \check{Q}} \size^k(\lambda)\] from the quasipolyomial $\ac^{\check{\Lambda}}_{\sizeshift^k}(b)$, we use the results of~\Cref{sec:symmetry_of_digram}: the coroot lattice $\check{Q}$ is a lattice of index $f$ inside $\check{\Lambda}$, with a group $\cycl$ of order $f$ acting freely (\Cref{thm:cycl_action_and_Q}).  By~\Cref{thm:cycl_action_and_quad}, $\size$ is invariant under the action of $b\cycl$.  Specifically, we have the simple relationship \[\frac{1}{f}\sum_{\lambda \in b \ac \cap \check{\Lambda}} \sizeshift^k(x) = \sum_{\lambda \in  b \ac \cap \check{Q}} \sizeshift^k(x).\]

We are therefore now in the desirable position of needing to collect enough points to fully determine the polynomial $\ac^{\check{\Lambda}}_{\sizeshift^k}(b)_j$, for all $j$ coprime to $h$.

\begin{remark}
We note here that since $\left|b\ac \cap \check{\Lambda} \right|=\left|(h+b)\ac^\circ \cap \check{\Lambda}\right|$, by Part (2) of~\Cref{thm:generized_ehrhart} we have that $\ac^{\check{\Lambda}}_{\sizeshift^k}(b)=\ac^{\check{\Lambda}}_{\sizeshift^k}(-h-b).$  Thus, for each point $b$ for which we can evaluate $\ac^{\check{\Lambda}}_{\sizeshift^k}(b)$, we get the second point $-h-b$ ``for free.''
\label{rem:points_for_free}
\end{remark}

By~\Cref{eq:weights_inside_ba}, the points $b\ac \cap \check{\Lambda}$ are the positive integral solutions to the linear equation \[\sum_{i=0}^n c_i x_i = b, \hspace{5em} 0 \leq x_i \in \mathbb{Z}.\]  Restricting to types $\widetilde{A}_n$ and $\widetilde{D}_n$, by fixing a small $b$ but letting $n$ be arbitrary, we can explicitly describe these points and sum $\sizeshift^k$ over this description for all $n$ simultaneously. 

Over the next two sections, we compute by hand the variance in type $\widetilde{A}_n$ (\Cref{thm:variance}) and expected value in type $\widetilde{D}_n$ (\Cref{thm:expected_size}), after which we discuss automation and the computation for type $\widetilde{E}_n$. %

\subsection{\Cref{thm:us_variance}, or \Cref{thm:variance} in type $\widetilde{A}_n$}
\label{sec:type_a_variance}

In type $\widetilde{A}_n$, $m(\widetilde{A}_n)=1$, and so $\ac^{\check{\Lambda}}_{\sizeshift^k}(b)=\ac^{\check{\Lambda}}_{\sizeshift^k}(b)_0$ is a \emph{polynomial}.  We have the relation of polynomials \[\ac^{\check{Q}}_{\sizeshift^2}(b) = \frac{1}{n+1}\ac^{\check{\Lambda}}_{\sizeshift^2}(b),\]  since this equality holds for all $b$ coprime to $h$, and therefore for all $b$.  We may therefore choose $b$ without concern as to its residue class modulo $h$.

All exponents $e_1,e_2,\ldots,e_n$ are coprime to $h$ and---as they are less than $h$---have the property that integral dilations $e_i \ac$ do not contain any interior lattice points.  P.~Johnson used these properties along with Ehrhart reciprocity to identify $n$ zeroes of the polynomial $\ac^{\check{\Lambda}}_{\sizeshift}(b)$ of degree $n+2$~\cite[Corollary 3.8]{johnson2015lattice}.

Furthermore, it is easy to see that $\ac^{\check{\Lambda}}_{\sizeshift}(1)=0$, so that by~\Cref{rem:points_for_free} we also have $\ac^{\check{\Lambda}}_{\sizeshift}(-h-1)=0$.  This gives $(n+2)$ zeroes of $g(b)$, and it remains only to check that the constant term $\ac^{\check{\Lambda}}_{\sizeshift}(0)=\sizeshift(0)=-\frac{n(h+1)}{24}$ by~\Cref{thm:explicit_quadr}. 
\medskip

To compute the variance, we will evaluate

    \[\ac^{\check{\Lambda}}_{\sizeshift^2}(b):=\sum_{\lambda \in b\ac \cap \check{\Lambda}} \sizeshift(\lambda)^2,\]

which is a polynomial of degree $n+4$, by~\Cref{thm:generized_ehrhart}.  The same reasoning as above gives us $(n+2)$ zeroes of $v(b)$, as well as the constant term

    \[\ac^{\check{\Lambda}}_{\sizeshift^2}(0)=\left(\frac{n(h+1)}{24}\right)^2.\]

We now have

    \begin{equation}\ac^{\check{\Lambda}}_{\sizeshift^2}(b)=\left(\prod_{i=1}^n(b+i) \right)(b-1)(b+h+1)(b-b_1)(b-b_2)c,\label{eq:poly_a_var}\end{equation}

and we have identified the value of $\ac^{\check{\Lambda}}_{\sizeshift^2}(b)$ at $b=0$.

We require two additional points, which we obtain in the next two subsections by explicitly evaluating $\ac^{\check{\Lambda}}_{\sizeshift^2}(2)$ separately for $\widetilde{A}_{n}$ with $n$ even and $n$ odd.

There are $\frac{(n+1)(n+2)}{2}$ coweight points in $2\ac$.  These are given explicitly as follows, where each line corresponds to a $\cycl$-orbit of coweights.
\begin{itemize}
    \item $w_i+w_j,w_{i+1}+w_{j+1},\ldots,w_{i+n-1}+w_{j-1}$, for $0 \leq i \leq j \leq n$.
\end{itemize}

\subsubsection{$n=0 \mod 2$}
In this case $h=n+1$ is coprime to $2$, and for each $1 \leq i \leq \frac{n+2}{2}$, there are $(n+1)$ points $\mu \in 2\ac\cap\check{\Lambda}$ with $\sizeshift(\mu) = \binom{i}{2}$.

\begin{remark}
We can see this combinatorially, by noting (as in the introduction) that the set of $2$-cores consists of exactly those partitions of staircase shape $(k,k-1,\ldots,1)$ for $k \in \mathbb{N}$, along with the empty partition.  When $n+1$ is odd, the simultaneous $(2,n+1)$ cores will then be those $2$-cores with fewer than $\frac{n+2}{2}$ rows.  
\end{remark}

We compute that \begin{equation}\ac^{\check{\Lambda}}_{\sizeshift^2}(2)=(n+1) \sum_{i=1}^{\frac{n+2}{2}} \binom{i}{2}^2 = \frac{(3n^2+12n+4)(n + 4)(n+2)(n+1)(n)}{1920}.\label{eq:a0zise2}\end{equation}

\subsubsection{$n=1 \mod 2$}

In this case there are the same $\frac{(n+1)(n+2)}{2}$ coweight points in $2\ac$ as before, but the evaluation of $\size$ on these points changes (since $2$ is not relatively prime to $h$, we no longer have the combinatorial interpertation as a sum over $2$-cores).  In particular, one can check that there are

\begin{itemize}
    \item $\frac{n+1}{2}$ coweight points $\mu \in 2\ac \cap \check{\Lambda}$ with $\sizeshift(\mu)=-\frac{1}{8}$, and
    \item for $1 \leq i \leq \frac{n+1}{2}$, there are $n+1$ coweight points $\mu \in 2\ac \cap \check{\Lambda}$ with $\sizeshift(\mu)=\frac{4i^2-1}{8}$.
\end{itemize}

We compute that \begin{align} \ac^{\check{\Lambda}}_{\sizeshift^2}(2)&=(n+1)\left(\frac{1}{2}\cdot \left(\frac{-1}{8}\right)^2+\sum_{i=1}^{\frac{n+1}{2}}\left(\frac{4i^2-1}{8}\right)^2 \right) \\ &= \frac{(3n^2+12n+4)(n+4)(n+2)(n+1)n}{1920}.\label{eq:a0zise22}\end{align}

\bigskip

Comparing~\Cref{eq:a0zise22} with~\Cref{eq:a0zise2}, we see that the formula for $ \ac^{\check{\Lambda}}_{\sizeshift^2}(2)$ does not depend on the parity of $n$.   By Ehrhart reciprocity, we have also found the value of $ \ac^{\check{\Lambda}}_{\sizeshift^2}(-(n+3))$.  A straightforward computation with~\Cref{eq:poly_a_var} now yields that

    \begin{equation}
    (b-b_1)(b-b_2)c = \frac{n(n+1)(2b+2b^2-10n+9bn+7b^2n-5n^2+7bn^2)}{2880|A_n|}.
    \label{eq:Aneqzeromod2}
    \end{equation}

\medskip

Using the relation $\mathbb{V}(X)=\mathbb{E}(X^2)-\mathbb{E}(X)^2$,~\Cref{eq:Aneqzeromod2}, and~\Cref{thm:johnson_expected_size}, we conclude the following theorem after substituting $a=n+1$.

{
\renewcommand{\thetheorem}{\ref{thm:us_variance}}
\begin{theorem}[{\Cref{thm:variance} in type $\widetilde{A}_n$: Variance of $\size$ on $(a,b)$-cores}]
$ $\newline
	For $\gcd(a,b)=1$, \[\Var{\lambda \in \core(a,b)}{\size(\lambda)} = \frac{ab(a-1)(b-1)(a+b)(a+b+1)}{1440}.\]
\end{theorem}
\addtocounter{theorem}{-1}
}

\subsection{\Cref{thm:expected_size} in type $\widetilde{D}_n$}
\label{sec:type_d_expected}
In type $\widetilde{D}_{n}$, the period $m(\widetilde{D}_n)=2$, and so $\ac^{\check{\Lambda}}_{\sizeshift}(b)$ has two polynomial components, $\ac^{\check{\Lambda}}_{\sizeshift}(b)_0$ and $\ac^{\check{\Lambda}}_{\sizeshift}(b)_1$.  As the Coxeter number $h=2n-2$ is even, we are interested only in $\ac^{\check{\Lambda}}_{\sizeshift}(b)_1$.  We may therefore only choose odd $b$ when trying to determine the desired component of the quasipolynomial.

There are $n-1$ \emph{distinct} odd exponents, which lie in the desired residue class modulo $2$ (since they are coprime to $h=2n-2$).  As in~\Cref{sec:type_a_variance}, by considering the dilations $e_i \ac$ for these $n-1$ exponents, we identify $n-1$ zeroes of the $1 \mod 2$ component of the polynomial $\ac^{\check{\Lambda}}_{\sizeshift}(b)_1$.  It is similarly easy to evaluate $\ac^{\check{\Lambda}}_{\sizeshift}(1)_1=0$, so that also $\ac^{\check{\Lambda}}_{\sizeshift}(-2n+1)_1=0$.  We therefore have found $(n+1)$ zeroes of $\ac^{\check{\Lambda}}_{\sizeshift}(b)_1$, and can write

    \begin{equation}\ac^{\check{\Lambda}}_{\sizeshift}(b)_1 = \left(\prod_{i=1}^{n-1} (b+2i-1)\right)(b-1)(b+2n-1)(b+b_1)c,\label{eq:dpolyexp}\end{equation}
    
so that we require two additional points to find the unknowns $b_1$ and $c$ and fully determine $\ac^{\check{\Lambda}}_{\sizeshift}(b)_1$.

Sadly, the remaining exponent is equal to $n-1$, which is either \emph{repeated} (when $n$ is even) or \emph{even} (when $n$ is odd).  We are therefore unable to use this exponent to find a zero of $\ac^{\check{\Lambda}}_{\sizeshift}(b)_1$.  Furthermore, we cannot use the evaluation $\ac^{\check{\Lambda}}_{\sizeshift}(0)_1$, since the dilation of the factor $b=0$ is not in the desired $1 \mod 2$ residue class.

We are left with no choice but to compute $\ac^{\check{\Lambda}}_{\sizeshift}(b)_1$ for some additional odd value of $b$.  The smallest unidentified such $b$ is $3$, and we now determine $\ac^{\check{\Lambda}}_{\sizeshift}(3)_1$.

\medskip

There are $4(n+2)$ coweight points inside $3\ac$, and $(n+2)$ corresponding coroot points.  These coweights are given explicitly as follows, where each line corresponds to a $\cycl$-orbit of coweights.

There are $20$ coweight points in $3\ac$ arranged in $\cycl$ orbits of size $4$:
\begin{itemize}
    \item $3w_0, 3w_1, 3w_n, 3w_{n-1}$,
    \item $2w_0+w_1,w_0+2w_1,2w_{n-1}+w_n,w_{n-1}+2w_n$,
    \item $w_0+w_{n-1}+w_n,w_0+w_1+w_n,w_0+w_1+w_{n-1},w_1+w_{n-1}+w_n$,
    \item $w_0+2w_n,2w_0+w_{n-1},2w_1+w_n,w_1+2w_{n-1}$,
    \item $2w_0+w_n,w_0+2w_{n-1},2w_1+w_{n-1},w_1+2w_n$.
\end{itemize}

There are an additional $4(n-3)$ coweight points of the form
\begin{itemize}
    \item $w_0+w_i, w_{n-i}+w_n, w_1+w_i, w_{n-i}+w_{n-1}$ for $2\leq i \leq n-2$,
\end{itemize}

\subsubsection{$n\neq 1 \mod 3$}
In this case, $3$ is coprime to $h=2n-2$.  

\medskip

Let $\pent(i):=\frac{1}{2}\left(3 \left\lfloor \frac{i+1}{2}\right\rfloor^2+(-1)^i\left\lfloor \frac{i+1}{2}\right\rfloor\right)$ be the $i$th largest pentagonal number.

One can check that there are 
\begin{itemize}
    \item four coweight point $\mu$ with $\zise(\mu)=\pent(i)$ for $0 \leq i \leq n-2-\lfloor\frac{n-1}{3}\rfloor$;
	\item eight coweight points $\mu$ with $\zise(\mu)=\pent\left(n-1-\lfloor\frac{n-1}{3}\rfloor\right)$; and
	\item alternatingly four or zero coweight points (starting with four) $\mu$ with $\zise(\mu)=\pent(i)$ for $n-\lfloor\frac{n-1}{3}\rfloor \leq i \leq 2n-2-2\lfloor\frac{n-1}{3}\rfloor$.
\end{itemize}

Thus, \tiny \[\ac^{\check{\Lambda}}_{\sizeshift}(3)_1=\sum_{i=0}^{n-2-\lfloor\frac{n-1}{3}\rfloor} 4\cdot\pent(i) + 8\cdot\pent\left(n-1-\left\lfloor\frac{n-1}{3}\right\rfloor\right)+\sum_{i=n-\lfloor\frac{n-1}{3}\rfloor}^{2n-2-2\lfloor\frac{n-1}{3}\rfloor} 4\cdot\frac{1+(-1)^{i-n+\lfloor\frac{n-1}{3}\rfloor}}{2}\pent(i).\] \normalsize

This expression simplifies to

\begin{equation} \ac^{\check{\Lambda}}_{\sizeshift}(3)_1=4\cdot\frac{n(n+1)(n+2)}{6}.\label{eq:a0dnexp}\end{equation}

\subsubsection{$n= 1 \mod 3$}

Let $\bino(i):=3\binom{i+1}{2}+\frac{1}{3}$, so that $\bino(0)=\frac{1}{3}$.  We check that there are

\begin{itemize}
\item eight coweight points with $\size(x)=\bino(i)$ for $0 \leq i \leq \frac{n-4}{3}$;

\item 12 coweight points with $\size(x)=\bino\left(\frac{n-1}{3}\right)$; and

\item four coweight points with $\size(x)=\bino(i)$ for $\frac{n+2}{3} \leq i \leq \frac{2n-2}{3}.$
\end{itemize}

Thus,

\[\ac^{\check{\Lambda}}_{\sizeshift}(3)_1 = \sum_{i=0}^{\frac{n-4}{3}}8\cdot\bino(i)+12\cdot \bino\left(\frac{n-1}{3}\right)+\sum_{i=\frac{n+2}{3}}^{\frac{2n-2}{3}}4\cdot\bino(i).\]

As before, this expression simplifies to

\begin{equation}\ac^{\check{\Lambda}}_{\sizeshift}(3)_1 = 4\cdot \frac{n(n+1)(n+2)}{6}.\label{eq:a0dnexp2}\end{equation}

\bigskip

Comparing~\Cref{eq:a0dnexp2} with~\Cref{eq:a0dnexp}, we see that the formula for $ \ac^{\check{\Lambda}}_{\sizeshift}(3)_1$ does not depend on the residue class of $n \mod 3$.   By Ehrhart reciprocity, this also determines the value of $\ac^{\check{\Lambda}}_{\sizeshift}(-2n-1)_1$.  A straightforward computation with~\Cref{eq:dpolyexp} now yields~\Cref{thm:expected_size}.

{
\renewcommand{\thetheorem}{\ref{thm:expected_size}}
\begin{theorem}[{Expected $\size$ of a $(\widetilde{D}_n,b)$-core}]
$ $\newline
	For $\gcd(h,b)=1$, \[\Expt{\lambda \in \core(\widetilde{D}_n,b)}{\size(\lambda)} = \frac{n(b-1)(b+h+1)}{24}.\]
\end{theorem}
\addtocounter{theorem}{-1}
}

\subsection{Automation:~\Cref{thm:us_third_moment}, and~\Cref{thm:variance} in type $\widetilde{D}_n$}
\label{sec:automation}

We now describe how we automated the computations to compute the third moment in type $\widetilde{A}_n$ (\Cref{thm:us_third_moment}) and variance in type $\widetilde{D}_n$ (\Cref{thm:variance}). 

Let $C_\Phi = \left(\langle \alpha_i,\alpha_j,\rangle \right)_{1\leq i,j \leq n}$ be the Cartan matrix for the root system $\Phi$.  Then, if $x=\sum_{i=1}^n x_i w_i$ is expressed in terms of the coweight basis, \[\left\|x\right\| = (x_1,\ldots,x_n)^T \cdot C_\Phi^{-1}\cdot (x_1,\ldots,x_n).\]

\begin{proposition}[{\cite[Table 1]{humphreys1972introduction}}]
The $(i,j)$th entry of $C_\Phi^{-1}$ is given
\begin{itemize}
    \item in type $A_n$ by \[C^{-1}_{i,j} = \begin{cases} \frac{i(n+1-j)}{n+1} & \text{ if } i \leq j, \\ \frac{j(n+1-i)}{n+1} & \text{ otherwise;} \end{cases} \hspace{3em} \text{and}\]

    \item in type $D_n$ by \[C^{-1}_{i,j} = \begin{cases} \min(i,j) & \text{ if } i, j \leq n-2,  \\ \frac{\min(i,j)}{2} & \text{ if } \max(i,j)>n-2 \text{ and } \min(i,j)\leq n-2, \\ \frac{n}{4} & \text{ if } n-1\leq i=j, \\ \frac{n-2}{4} & \text{ otherwise.} \end{cases}\]
\end{itemize}
\label{prop:inv_of_cartan}
\end{proposition}

\begin{proposition}
The difference $\left\|x-\frac{b}{h}\rho\right\|-\|x\|$ is a linear function of the $x_i$ given
\begin{itemize}
    \item in type $\widetilde{A}_n$ by \[\left\|x-\frac{b}{h}\rho\right\|=\|x\|+\left(\frac{b}{n+1}\right)^2\frac{1}{2}\binom{n+2}{3}-\sum_{i=1}^n\frac{b i(n+1-i)}{n+1}x_i; \hspace{3em} \text{and}\]
    \item in type $\widetilde{D}_n$ by \[\left\|x-\frac{b}{h}\rho\right\|=\|x\|+\left(\frac{b}{2n-2}\right)^2\frac{1}{2n+1}\binom{2n+1}{4} -\left(\sum_{i=1}^{n-2} \frac{b i (2n-1-i)}{2n-2} x_i\right)+\frac{bn}{4}(x_{n-1}+x_n).\]
\end{itemize}
\label{prop:rho_difference}
\end{proposition}

\begin{proof}
This follows from direct computation with~\Cref{prop:inv_of_cartan}.
\end{proof}

We now describe our automation of the calculations in~\Cref{sec:type_a_variance,sec:type_d_expected}, taking type $\widetilde{A}_n$ as our example.  In type $\widetilde{A}_n$, the coweight points $(x_1,x_2,\cdots,x_n)$ (where $x_i$ is the $i$th coordinate in the coweight basis) contained in $b\ac$ are exactly the nonnegative solutions to the linear equation \begin{equation}\sum_{i=0}^n x_i = b.\label{eq:type_a_linear}\end{equation}

Let $\comp(b):=\{c=(c_1,c_2,\ldots,c_{\ell(c)}) : \sum_{i=1}^{\ell{c}} c_{\ell(c)} = b\}$ be all compositions of $b$.  For a fixed dilation factor $b$,~\Cref{eq:type_a_linear} ensures that there will only be at most $b$ nonzero coordinates $x_i$ when computing $|x|$.  By~\Cref{prop:rho_difference}, we can calculate $\zise(x)=\frac{h}{2}\left\|x-\frac{b}{h}\rho\right\|-\frac{n(h+1)}{24}$ from $\|x\|$ and a linear function in $x$.  Let $\mu:=\Expt{\lambda \in b\ac\cap\check{\Lambda}}{\sizeshift(\lambda)}$.  Using the explicit formulas for $C^{-1}_{i,j}$, for $c \in \comp(b)$ the summation \[\sum_{0\leq i_1< i_2 < \cdots < i_{\ell(c)} \leq n} \left(\zise\left(\sum_{j=1}^{\ell(c)}c_{i_j}w_{i_j}\right)-\mu\right)^k\] for $k \geq 1$ may therefore be explicitly evaluated (either by hand, or by computer) as a \emph{polynomial} of degree $2k+\ell(c)$.  Summing now over all $2^{b-1}$ compositions for a fixed $b$, we can determine the polynomial of degree $2k+b$
\[\ac_{(\zise-\mu)^k}^{\check{\Lambda}}(\rat)=\sum_{c \in \comp(b)} \sum_{0\leq i_1< i_2 < \cdots < i_{\ell(c)} \leq n} \left(\zise\left(\sum_{j=1}^{\ell(c)}c_{i_j}w_{i_j}\right)-\mu\right)^k.\]

One can write a similar sum in type $\widetilde{D}_n$, treating the four simple roots in the orbit of the affine node differently from the rest.  
We wrote {\sf Mathematica} code to find $\ac_{(\zise-\mu)^3}^{\check{\Lambda}}(\rat)$ for $\rat=2,3,4,5$ simultaneously for all ranks $n$ and determine the third moment in type $\widetilde{A}_n$~\cite{math}.
{
\renewcommand{\thetheorem}{\ref{thm:us_third_moment}}
\begin{theorem}[{Third moment of $\size$ on $(a,b)$-cores}]
$ $\newline
For $\gcd(a,b)=1$, let $\mu:=\Expt{\lambda \in \core(a,b)}{\size(\lambda)}$.  Then \tiny\[\sum_{\lambda \in \core(a,b)} \left(\size(\lambda)-\mu\right)^3 = \frac{ab(a-1) (b-1) (a+b) (a+b+1) \left(2 a^2 b-3 a^2+2 a b^2-3 a b-3 b^2-3 \right)}{60480}.\]\normalsize
\end{theorem}
\addtocounter{theorem}{-1}
}

We used similar code in type $D_n$ to compute $\ac_{(\zise-\mu)^2}^{\check{\Lambda}}(\rat)$ for $\rat = 3,5$ for all ranks $n$ to determine variance.  
{
\renewcommand{\thetheorem}{\ref{thm:variance}}
\begin{theorem}[{Variance of $\size$ on $(\widetilde{D}_n,b)$-cores}]
$ $\newline
	For $\gcd(h,b)=1$, \[\Var{\lambda \in \core(\widetilde{D}_n,b)}{\size(\lambda)} = \frac{nhb(b-1)(h+b)(h+b+1)}{1440}.\]
\end{theorem}
\addtocounter{theorem}{-1}
}

\subsection{Automation: \Cref{thm:expected_size,thm:variance} in types $\widetilde{E}_n$}
\label{sec:e_series}


In the exceptional types $\widetilde{E}_6, \widetilde{E}_7,$ and $\widetilde{E}_8$,~\Cref{thm:expected_size,thm:variance} are a finite check, which we accomplish using a similar method as in~\Cref{sec:automation} with the freely available program {\sf Normaliz}~\cite{bruns2001normaliz,bruns2010normaliz}.

Suppose $\Phi$ is a root system of type $E_6$, $E_7$ or $E_8$.
Using the fact that $\Phi$ is simply laced (and therefore $g=h$ and $\rho=\check{\rho}$) we calculate that 
\[
 \sum_{\lambda\in\core(\wa,b)}\size(\lambda)=\sum_{\lambda\in\rat\ac\cap\Q}\zise(\lambda)=\sum_{\lambda\in\rat\ac\cap\Q}\left(\frac{h}{2}\|\lambda\|^2-\rat\langle\lambda,\check{\rho}\rangle+(\rat^2-1)\frac{n(h+1)}{24}\right).\]

Thus our task is to calculate the Euler-Maclaurin quasipolynomials 
$\sum_{\lambda\in\rat\ac\cap\Q}\frac{h}{2}\|\lambda\|^2$ and $\sum_{\lambda\in\rat\ac\cap\Q}\langle\lambda,\check{\rho}\rangle$.
To be able to use {\sf Normaliz} for this task, we need to interpret these sums as sums over $\Z^n$ as follows. 
Let $A=(\langle\check{\alpha}_i,\alpha_j\rangle)_{ij}$ be the Cartan matrix of $\Phi$. Write $\tilde{\alpha}=\sum_{i=1}^nc_i\alpha_i$ and let $c=(c_1,c_2,\ldots,c_n)$ be the row vector of coefficients.
\[
 \sum_{\lambda\in\rat\ac\cap\Q}\frac{h}{2}\|\lambda\|^2=\sum_{\substack{(x_1,x_2,\ldots,x_n)\in\Z^n\\\sum_{i=1}^nx_i\check{\alpha}_i\in\rat\ac}}\frac{h}{2}\left\|\sum_{i=1}^nx_i\check{\alpha}_i\right\|^2
 =\sum_{\substack{x=(x_1,x_2,\ldots,x_n)\in\Z^n\\Ax\geq0\\\rat-c^TAx\geq0}}\frac{h}{2}x^TAx.
\]

Similarly,
\begin{align*}
 \sum_{\lambda\in\rat\ac\cap\Q}\langle\lambda,\check{\rho}\rangle=\sum_{\substack{x=(x_1,x_2,\ldots,x_n)\in\Z^n\\Ax\geq0\\\rat-c^TAx\geq0}}\sum_{i=1}^nx_i.
\end{align*}
For the purposes of using Normaliz, it is helpful to replace the set 
\[\{(x_1,x_2,\ldots,x_n)\in\Z^n:Ax\geq0\text{ and }\rat-c^TAx\geq0\}\]
with
\[\{(x_1,x_2,\ldots,x_n,x_{n+1})\in\Z^{n+1}:Ax\geq0,\text{ }x_{n+1}-c^TAx\geq0\text{ and }\lambda(x)=\rat\},\]
where $\lambda$ is the linear functional on $\R^{n+1}$ defined by $\lambda(x)=x_{n+1}$ for all $x\in\R^{n+1}$.
The linear functional $\lambda$ is thus used as a grading.\\
\\

As an example calculation in type $E_6$, in the visual interface {sf jNormaliz}, we input the matrix
\[ M=\left( \begin{array}{ccccccc}
2&0&-1&0&0&0&0\\
0&2&0&-1&0&0&0\\
-1&0&2&-1&0&0&0\\
0&-1&-1&2&-1&0&0\\
0&0&0&-1&2&-1&0\\
0&0&0&0&-1&2&0\\
0&-1&0&0&0&0&1\\
\end{array} \right)\] 
and the grading $(0,0,0,0,0,0,1)$. Then we use its generalized Ehrhart series functionality with the quadratic polynomial $\frac{h}{2}x^TAx$ to find 
\begin{align*}
 \sum_{\lambda\in\rat\ac\cap\Q}\frac{h}{2}\|\lambda\|^2&=\sum_{\substack{x=(x_1,x_2,\ldots,x_n)\in\Z^n\\Ax\geq0\\\rat-c^TAx\geq0}}\frac{h}{2}x^TAx.
\end{align*}
as a quasipolynomial in $\rat$ with period $6$. Similarly we calculate
\begin{align*}
 \sum_{\lambda\in\rat\ac\cap\Q}\langle\lambda,\check{\rho}\rangle=\sum_{\substack{x=(x_1,x_2,\ldots,x_n)\in\Z^n\\Ax\geq0\\\rat-c^TAx\geq0}}\sum_{i=1}^nx_i.
\end{align*}
as a quasipolynomial in $\rat$ with period $6$. Somewhat miraculously, we find that 
\begin{align*}
 \sum_{\lambda\in\rat\ac\cap\Q}\left(\frac{h}{2}\|\lambda\|^2-\rat\langle\lambda,\check{\rho}\rangle+(\rat^2-1)\frac{n(h+1)}{24}\right)\\
\end{align*}
is a \emph{polynomial} in $\rat$ equal to
\begin{multline*}
 \frac{1}{207360}(\rat-1)(\rat+1)(\rat+4)(\rat+5)(\rat+7)(\rat+8)(\rat+11)(\rat+13)\\
 =\frac{1}{24}n(\rat-1)(\rat+h+1)\times\frac{1}{|W|}\prod_{i=1}^n(\rat+e_i),
\end{multline*}
proving~\Cref{thm:expected_size} in type $\widetilde{E}_6$.

\section{Open Problems}
\label{sec:open_problems}

In this section, we present some open problems and conjectures.  \Cref{sec:integrals} proposes a first step towards finding formulas beyond the third moment.  \Cref{sec:non_simply} then asks if it is possible to extend our theorems to non-simply-laced types.  Finally,~\Cref{sec:comb_models} suggests that existing combinatorics associated to the representation theory of affine Lie algebras might to be harnessed to develop combinatorial models for cores in other types.

\subsection{Higher Moments and Integrals}
\label{sec:integrals}

It is natural to ask about formulas for higher moments, even though we do not believe there are uniform formulas (see~\Cref{rem:no_uniform}).

\begin{openproblem}
Extend~\Cref{thm:max_size,thm:expected_size,thm:variance} to higher moments.
\label{problem:extend_to_higher_moments}
\end{openproblem}

Rather than compute the \emph{entire} Ehrhart quasipolynomial $\ac^{\check{Q}}_{(\sizeshift-\mu)^k}(b)$, we could instead ask for its leading coefficient.  In general, this leading coefficient turns out to be an integral over the polytope.

\begin{theorem}[~\cite{baldoni2011integrate,baldoni2012computation}]
    Fix $\mathcal{P},L,$ and $h$ as in~\Cref{sec:weighted_ehrhart}.  The leading coefficient of the Ehrhart quasipolynomial  $\mathcal{P}^L_h(b)$ does not depend on $b$ and is given by \[ \int_{x \in \mathcal{P}} h(x).\]
\end{theorem}

For example, the leading coefficients of~\Cref{thm:expected_size,thm:variance} give the following formulae for the integral of the quadratic form $\size$.  Here we have normalized so that $\mathrm{Vol}(b\ac) = 1$.

\begin{corollary}  For $\gcd(h,b)=1$,
\begin{align*}
    \int_{x \in \Sommers_\Phi(b)} \size(x) &= \frac{nb^2}{24}, \text{ and}\\
    \int_{x \in \Sommers_\Phi(b)} \left(\size(x)-\frac{nb^2}{24}\right)^2 &= \frac{nhb^4}{1440}.\\
\end{align*}
\end{corollary}

Although~\Cref{rem:no_uniform} suggests that there is no uniform formula for the third moment, we conjecture that the leading coefficient of the third moment \emph{does} have a uniform formula.

\begin{conjecture}
For $\gcd(h,b)=1$, we have the following uniform integral. \[\int_{x \in \Sommers_\Phi(b)} \left(\size(x) - \frac{nb^2}{24}\right)^3 = \frac{nhb^6}{60480}(2(h-1)-1).\]
\end{conjecture}

Computational evidence suggests that even this leading coefficient lacks a uniform formula beyond the third moment.  We record a few of these leading coefficients here.

\begin{conjecture}
Let the leading coefficient of $\ac^{\check{Q}}_{(\sizeshift-\mu)^k}(b)$ be denoted by \[\mathrm{top}_\Phi^b(i):= \int_{x \in \Sommers_{\Phi}(b)} \left(\size(x) - \frac{nb^2}{24}\right)^i.\]

In type $\widetilde{A}_n$,
\begin{align*}
     \mathrm{top}_\Phi^b(4) &= \frac{nhb^8}{4838400} \, {\left(19 \, n^{2} - 13 \, n + 4\right)}, \\
    \mathrm{top}_\Phi^b(5) &=  \frac{nhb^{10}}{95800320} \, {\left(23 \, n^2 - 25 \, n + 12\right)} {\left(2
\, n - 1\right)},\\
    \mathrm{top}_\Phi^b(6) &= {\scriptstyle \frac{nhb^{12}}{4184557977600} \, {\left(307561 \, n^{4} - 826062 \, n^{3} +
1048509 \, n^{2} - 647948 \, n + 155040\right)}}, \text{ and}\\
\mathrm{top}_\Phi^b(7) &= {\scriptstyle
\frac{nhb^{14}}{1195587993600} \, {\left(15562 \, n^{5} - 64721 \, n^{4} +
129288 \, n^{3} - 142241 \, n^{2} + 82300 \, n - 19488\right)}}.
\end{align*}
    
In type $\widetilde{D}_n$,
\begin{align*}
\mathrm{top}_\Phi^b(4) &= \frac{nhb^8}{2419200}(31n^2-99n+86),\\
\mathrm{top}_\Phi^b(5) &=        	
\frac{nhb^{10}}{23950080} \, {\left(70 \, n^{3} - 365 \, n^{2} + 667 \, n -
426\right)}, \text{ and}\\
\mathrm{top}_\Phi^b(6) &=
{\scriptstyle \frac{nhb^{12}}{523069747200} \, {\left(859445 \, n^{4} - 6449250 \, n^{3} +
19050243 \, n^{2} - 26075294 \, n + 13852536\right)}}.
\end{align*}
\end{conjecture}

We do not even have a conjecture for the denominators of these expressions, although D.~Armstrong has suggested a connection to the Dirichlet $\eta$ function~\cite{drewprivate}.

\subsection{Non-Simply-Laced Types}
\label{sec:non_simply}

It is reasonable to ask for analogues of our results in non-simply-laced types.

\begin{openproblem}
Extend~\Cref{thm:max_size,thm:expected_size,thm:variance} to non-simply-laced types.
\label{problem:extend}
\end{openproblem}

To whet the reader's appetite, we offer a conjecture for the Fu{\ss}-Catalan case $b=mh+1$ in type $\widetilde{C}_n$, for which our open problem seems to be low-hanging fruit (see~\Cref{rem:size_not_number}).

\begin{conjecture}
\[\Expt{\lambda \in \core(\widetilde{C}_n,mh+1)}{\size(\lambda)} =
\frac{m n (2(m+1)n^2+(m+3)n-(m+1))}{12}.\]
\end{conjecture}

\subsection{Basic Representations and Combinatorial Models}
\label{sec:comb_models}

In this section, we suggest that researchers interested in extending combinatorial models of cores to other types might benefit from existing combinatorial models arising in the representation theory of affine Lie algebras.

Fix an affine Lie algebra $\mathfrak{g}$.  The highest weight module $L(\Lambda_0)$ is the \defn{basic representation} of $\mathfrak{g}$.  We refer the reader to~\cite{kac1994infinite} for further details.  The module $L(\Lambda_0)$ has an associated crystal $\mathcal{B}(\Lambda_0)$, which is an infinite directed graph with a unique source (the highest weight), and with edges labeled by simple affine roots $\widetilde{\simp}$.  For $\alpha \in \widetilde{\simp}$, an \defn{$\alpha$-string} is a maximal connected chain of $\mathcal{B}(\Lambda_0)$ whose edges are all labeled by $\alpha$.  There is a $\wa$-action on the vertices of $B(\Lambda_0)$, where $s_{\alpha}$ acts by reversing all $\alpha$-strings.

\begin{theorem}[\cite{kac1994infinite}]
  The $\wa$-orbit of the highest weight in $\mathcal{B}(\Lambda_0)$ is in $\wa$-equivariant bijection with the coroot lattice $\check{Q}$.
  \label{thm:wa_orbit_is_weights}
\end{theorem}

Rather than redevelop the combinatorics of cores (or abaci) for other affine types---as in~\cite{hanusa2012abacus,beazley2015bijective}---we propose that it might be worthwhile to study the restriction of existing combinatorial models for the crystal $B(\Lambda_0)$ to the $\wa$-orbit of~\Cref{thm:wa_orbit_is_weights}.  For example, in type $\widetilde{A}_n$ we observe that the ``Young wall''  model illustrated for $n=2$ in~\cite[Figure 22]{kang2003crystal} recovers cores.

\section*{Acknowledgements}

We thank D.~Armstrong for helpful conversations and encouragement and P.~Johnson for clarifying certain details of his paper~\cite{johnson2015lattice}.  The second author thanks D.~Stanton for introducing him to cores, M.~Guay-Paquet for a conversation regarding the last equality of~\cite[Theorem 8.16]{macdonald1971affine}, and F.~Saliola for the use of his hardware for certain computations.

%
%

\newcommand{\etalchar}[1]{$^{#1}$}
\providecommand{\bysame}{\leavevmode\hbox to3em{\hrulefill}\thinspace}
\providecommand{\MR}{\relax\ifhmode\unskip\space\fi MR }
\providecommand{\MRhref}[2]{%
  \href{http://www.ams.org/mathscinet-getitem?mr=#1}{#2}
}
\providecommand{\href}[2]{#2}



\begin{thebibliography}{BBDL{\etalchar{+}}12}

\bibitem[AB14]{ardila2014double}
Federico Ardila and Erwan Brugalle, \emph{The double {G}romov-{W}itten
  invariants of {H}irzebruch surfaces are piecewise polynomial}, arXiv preprint
  arXiv:1412.4563 (2014).

\bibitem[Agg14a]{aggarwal2014armstrong}
Amol Aggarwal, \emph{{A}rmstrong's conjecture for $(k, mk+1)$-core partitions},
  arXiv preprint arXiv:1407.5134 (2014).

\bibitem[Agg14b]{aggarwal2014converse}
\bysame, \emph{A converse to vandehey's theorem on simultaneous core
  containment}, arXiv preprint arXiv:1408.0550 (2014).

\bibitem[Agg15]{aggarwal2015does}
\bysame, \emph{When does the set of $(a, b, c) $-core partitions have a unique
  maximal element?}, The Electronic Journal of Combinatorics \textbf{22}
  (2015), no.~2, P2--31.

\bibitem[AHJ14]{armstrong2014results}
Drew Armstrong, Christopher Hanusa, and Brant Jones, \emph{Results and
  conjectures on simultaneous core partitions}, European Journal of
  Combinatorics \textbf{41} (2014), 205--220.

\bibitem[AKS09]{aukerman2009simultaneous}
David Aukerman, Ben Kane, and Lawrence Sze, \emph{On simultaneous
  $s$-cores/$t$-cores}, Discrete Mathematics \textbf{309} (2009), no.~9,
  2712--2720.

\bibitem[AL14]{amdeberhan2014multi}
Tewodros Amdeberhan and Emily Leven, \emph{Multi-cores, posets, and lattice
  paths}, arXiv preprint arXiv:1406.2250 (2014).

\bibitem[Alp14]{alpoge2014self}
Levent Alpoge, \emph{Self-conjugate core partitions and modular forms}, Journal
  of Number Theory \textbf{140} (2014), 60--92.

\bibitem[ALW14]{armstrong2014rational}
Drew Armstrong, Nicholas~A Loehr, and Gregory~S Warrington, \emph{Rational
  parking functions and {C}atalan numbers}, arXiv preprint arXiv:1403.1845
  (2014).

\bibitem[And02]{anderson2002partitions}
Jaclyn Anderson, \emph{Partitions which are simultaneously $t_1$-and
  $t_2$-core}, Discrete Mathematics \textbf{248} (2002), no.~1, 237--243.

\bibitem[Arm15a]{armstrong2011conjecture}
Drew Armstrong, \emph{Rational {C}atalan combinatorics}, 2012 (accessed 12 May,
  2015).

\bibitem[Arm15b]{drewprivate}
\bysame, personal communication, 2015.

\bibitem[ARW13]{armstrong2013rational}
Drew Armstrong, Brendon Rhoades, and Nathan Williams, \emph{Rational
  associahedra and noncrossing partitions}, The Electronic Journal of
  Combinatorics \textbf{20} (2013), no.~3, P54.

\bibitem[Ath05]{athanasiadis2005refinement}
Christos Athanasiadis, \emph{On a refinement of the generalized {C}atalan
  numbers for {W}eyl groups}, Transactions of the American Mathematical Society
  \textbf{357} (2005), no.~1, 179--196.

\bibitem[Bar06]{barvinok2006computing}
Alexander Barvinok, \emph{Computing the {E}hrhart quasi-polynomial of a
  rational simplex}, Mathematics of Computation \textbf{75} (2006), no.~255,
  1449--1466.

\bibitem[BBDL{\etalchar{+}}11]{baldoni2011integrate}
Velleda Baldoni, Nicole Berline, Jesus De~Loera, Matthias K{\"o}ppe, and
  Mich{\`e}le Vergne, \emph{How to integrate a polynomial over a simplex},
  Mathematics of Computation \textbf{80} (2011), no.~273, 297--325.

\bibitem[BBDL{\etalchar{+}}12]{baldoni2012computation}
Velleda Baldoni, Nicole Berline, Jes{\'u}s~A De~Loera, Matthias K{\"o}ppe, and
  Mich{\`e}le Vergne, \emph{Computation of the highest coefficients of weighted
  {E}hrhart quasi-polynomials of rational polyhedra}, Foundations of
  Computational Mathematics \textbf{12} (2012), no.~4, 435--469.

\bibitem[BDF{\etalchar{+}}06]{baldwin2006self}
John Baldwin, Melissa Depweg, Ben Ford, Abraham Kunin, and Lawrence Sze,
  \emph{Self-conjugate $t$-core partitions, sums of squares, and $p$-blocks of
  $a_n$}, Journal of Algebra \textbf{297} (2006), no.~2, 438--452.

\bibitem[BGLX14]{bergeron2014compositional}
Francois Bergeron, Adriano Garsia, Emily Leven, and Guoce Xin,
  \emph{Compositional $(km, kn)$-shuffle conjectures}, arXiv preprint
  arXiv:1404.4616 (2014).

\bibitem[BIS10]{bruns2010normaliz}
Winfried Bruns, Bogdan Ichim, and Christof S{\"o}ger, \emph{Normaliz.
  {A}lgorithms for rational cones and affine monoids}, J. Algebra \textbf{324}
  (2010), 1098--1113.

\bibitem[BK01]{bruns2001normaliz}
Winfried Bruns and Robert Koch, \emph{Normaliz, a program to compute
  normalizations of semigroups}, Available by anonymous ftp from ftp.
  mathematik. Uni-Osnabrueck. DE/pub/osm/kommalg/software (2001).

\bibitem[BNP{\etalchar{+}}15]{beazley2015bijective}
Elizabeth Beazley, Margaret Nichols, Min~Hae Park, XiaoLin Shi, and Alexander
  Youcis, \emph{Bijective projections on parabolic quotients of affine {W}eyl
  groups}, Journal of Algebraic Combinatorics \textbf{41} (2015), no.~4,
  911--948.

\bibitem[BR47]{brauer1947conjecture}
Richard Brauer and Gilbert de~Beauregard Robinson, \emph{On a conjecture by
  {N}akayama}, Royal Society of Canada, 1947.

\bibitem[BR07]{beck2007computing}
Matthias Beck and Sinai Robins, \emph{Computing the continuous discretely:
  {I}nteger-point enumeration in polyhedra}, Springer Science \& Business
  Media, 2007.

\bibitem[BV07]{berline2007local}
Nicole Berline and Michele Vergne, \emph{Local {E}uler-{M}aclaurin formula for
  polytopes}, Mosc. Math. J \textbf{7} (2007), no.~3, 355--386.

\bibitem[CDH15]{ceballos2015combinatorics}
Cesar Ceballos, Tom Denton, and Christopher Hanusa, \emph{Combinatorics of the
  zeta map on rational {D}yck paths}, arXiv preprint arXiv:1504.06383 (2015).

\bibitem[CHW14]{chen2014average}
William Chen, Harry Huang, and Larry Wang, \emph{Average size of a
  self-conjugate $(s, t)$-core partition}, arXiv preprint arXiv:1405.2175
  (2014).

\bibitem[Fan96]{fan1996euler}
C~Kenneth Fan, \emph{Euler characteristic of certain affine flag varieties},
  Transformation Groups \textbf{1} (1996), no.~1-2, 35--39.

\bibitem[Fay11]{fayers2011t}
Matthew Fayers, \emph{The $t$-core of an $s$-core}, Journal of Combinatorial
  Theory, Series A \textbf{118} (2011), no.~5, 1525--1539.

\bibitem[Fay14]{fayers2014generalisation}
\bysame, \emph{A generalisation of core partitions}, Journal of Combinatorial
  Theory, Series A \textbf{127} (2014), 58--84.

\bibitem[Fay15]{fayers2015s}
\bysame, \emph{$(s,t)$-cores: a weighted version of {A}rmstrong's conjecture},
  arXiv preprint arXiv:1504.01681 (2015).

\bibitem[FMS09]{ford2009self}
Ben Ford, Ho{\`a}ng Mai, and Lawrence Sze, \emph{Self-conjugate simultaneous
  $p$-and $q$-core partitions and blocks of $a_n$}, Journal of Number Theory
  \textbf{129} (2009), no.~4, 858--865.

\bibitem[FV10]{fishel2010bijection}
Susanna Fishel and Monica Vazirani, \emph{A bijection between dominant {S}hi
  regions and core partitions}, European Journal of Combinatorics \textbf{31}
  (2010), no.~8, 2087--2101.

\bibitem[GKS90]{garvan1990cranks}
Frank Garvan, Dongsu Kim, and Dennis Stanton, \emph{Cranks and $t$-cores},
  Inventiones Mathematicae \textbf{101} (1990), no.~1, 1--17.

\bibitem[GMV14]{gorsky2014affine}
Eugene Gorsky, Mikhail Mazin, and Monica Vazirani, \emph{Affine permutations
  and rational slope parking functions}, arXiv preprint arXiv:1403.0303 (2014).

\bibitem[Hai94]{haiman1994conjectures}
Mark Haiman, \emph{Conjectures on the quotient ring by diagonal invariants},
  Journal of Algebraic Combinatorics \textbf{3} (1994), no.~1, 17--76.

\bibitem[HJ12]{hanusa2012abacus}
Christopher Hanusa and Brant Jones, \emph{Abacus models for parabolic quotients
  of affine {W}eyl groups}, Journal of Algebra \textbf{361} (2012), 134--162.

\bibitem[HN13]{hanusa2013number}
Christopher Hanusa and Rishi Nath, \emph{The number of self-conjugate core
  partitions}, Journal of Number Theory \textbf{133} (2013), no.~2, 751--768.

\bibitem[Hum72]{humphreys1972introduction}
James Humphreys, \emph{Introduction to {L}ie algebras and representation
  theory}, vol.~9, Springer Science \& Business Media, 1972.

\bibitem[IM65]{iwahori1965some}
Nagayoshi Iwahori and Hideya Matsumoto, \emph{On some {B}ruhat decomposition
  and the structure of the {H}ecke rings of $p$-adic {C}hevalley groups},
  Publications Math{\'e}matiques de l'IH{\'E}S \textbf{25} (1965), no.~1,
  5--48.

\bibitem[JK81]{james1981representation}
Gordon James and Adalbert Kerber, \emph{The representation theory of the
  symmetric group}, 1981.

\bibitem[Joh15]{johnson2015lattice}
Paul Johnson, \emph{Lattice points and simultaneous core partitions}, arXiv
  preprint arXiv:1502.07934 (2015).

\bibitem[Kac94]{kac1994infinite}
Victor~G Kac, \emph{Infinite-dimensional {L}ie algebras}, vol.~44, Cambridge
  university press, 1994.

\bibitem[Kan03]{kang2003crystal}
Seok-Jin Kang, \emph{Crystal bases for quantum affine algebras and
  combinatorics of {Y}oung walls}, Proceedings of the London Mathematical
  Society \textbf{86} (2003), no.~1, 29--69.

\bibitem[Kos76]{kostant1976macdonald}
Bertram Kostant, \emph{On macdonald's $\eta$-function formula, the {L}aplacian
  and generalized exponents}, Advances in Mathematics \textbf{20} (1976),
  no.~2, 179--212.

\bibitem[Las01]{lascoux2001ordering}
Alain Lascoux, \emph{Ordering the affine symmetric group}, Algebraic
  combinatorics and applications, Springer, 2001, pp.~219--231.

\bibitem[LP12]{lam2012alcoved}
Thomas Lam and Alexander Postnikov, \emph{Alcoved polytopes {II}}, arXiv
  preprint arXiv:1202.4015 (2012).

\bibitem[Mac71]{macdonald1971affine}
Ian~G Macdonald, \emph{Affine root systems and {D}edekind's $\eta$-function},
  Inventiones mathematicae \textbf{15} (1971), no.~2, 91--143.

\bibitem[Nak40]{nakayama1940some}
Tadasi Nakayama, \emph{On some modular properties of irreducible
  representations of a symmetric group, {I}, {II}}, Japan. J. Math \textbf{17}
  (1940), 165--184.

\bibitem[Nat08]{nath2008t}
Rishi Nath, \emph{On the $t$-core of an $s$-core partition}, Integers:
  Electronic Journal of Combinatorial Number Theory \textbf{8} (2008), no.~A28,
  A28.

\bibitem[Nat14]{nath2014symmetry}
\bysame, \emph{Symmetry in maximal $(s-1,s+1)$-cores}, arXiv preprint
  arXiv:1411.0339 (2014).

\bibitem[OS07]{olsson2007block}
J{\o}rn Olsson and Dennis Stanton, \emph{Block inclusions and cores of
  partitions}, Aequationes mathematicae \textbf{74} (2007), no.~1-2, 90--110.

\bibitem[Som05]{sommers2005b}
Eric~N Sommers, \emph{$b$-stable ideals in the nilradical of a {B}orel
  subalgebra}, Canadian mathematical bulletin \textbf{48} (2005), no.~3,
  460--472.

\bibitem[ST14]{sulzgruber2014type}
Robin Sulzgruber and Marko Thiel, \emph{Type $c$ parking functions and a zeta
  map}, arXiv preprint arXiv:1411.3885 (2014).

\bibitem[Sut98]{suter1998number}
Ruedi Suter, \emph{The number of lattice points in alcoves and the exponents of
  the finite {W}eyl groups}, Mathematics of computation (1998), 751--758.

\bibitem[SZ13]{stanley2013catalan}
Richard Stanley and Fabrizio Zanello, \emph{The {C}atalan case of {A}rmstrong's
  conjecture on core partitions}, arXiv preprint arXiv:1312.4352 (2013).

\bibitem[Thi15]{thiel2015anderson}
Marko Thiel, \emph{From {A}nderson to {Z}eta: A uniform bijection between {S}hi
  regions and the finite torus}, arXiv preprint arXiv:1504.07363 (2015).

\bibitem[Van08]{vandehey2008general}
Joseph Vandehey, \emph{A general theory of $(s, t)$-core partitions}, Ph.D.
  thesis, University of Oregon, 2008.

\bibitem[{Wol}]{math}
{Wolfram Research, Inc.}, \emph{Mathematica 8.0}.

\bibitem[Xio14]{xiong2014largest}
Huan Xiong, \emph{On the largest size of $(t,t+1,\ldots,t+p)$-core partitions},
  arXiv preprint arXiv:1410.2061 (2014).

\bibitem[YZZ14]{yang2014enumeration}
Jane Yang, Michael Zhong, and Robin Zhou, \emph{On the enumeration of
  $(s,s+1,s+2)$-core partitions}, arXiv preprint arXiv:1406.2583 (2014).

\end{thebibliography}

\end{document}